\documentclass[11pt]{article}
\usepackage[a4paper, total={6.5in, 9in}]{geometry}
\usepackage{latexsym}
\usepackage{amsmath}

\usepackage{amssymb}
\usepackage{amsthm}
\usepackage{float}
\usepackage{mathtools,mathrsfs}
\usepackage[dvipsnames]{xcolor}
\usepackage{hyperref}
\usepackage{bbm}
\usepackage{comment}
\usepackage{algorithm}
\usepackage{algpseudocode}
\usepackage{epsfig}
\usepackage[shortlabels]{enumitem}
\usepackage{mathtools, nccmath}
\usepackage[square,numbers]{natbib}
\usepackage{caption}
\usepackage{subcaption}
\DeclareMathOperator{\prob}{\mathbb{P}}
\DeclareMathOperator{\argmax}{argmax}

\newcommand{\calS}{{\mathcal{S}}}
\newcommand{\calA}{{\mathcal{A}}}
\newcommand{\calD}{{\mathcal{D}}}
\newcommand{{\calY}}{{\mathcal{Y}}}
\newcommand{\calH}{{\mathcal{H}}}
\newcommand{\calC}{{\mathcal{C}}}

\newcommand{\calF}{{\mathcal{F}}}

\newcommand{\R}{\mathbb{R}}
\newcommand{\E}{\mathbb{E}}

\newcommand{\dd}{{\rm d}}

\newcommand{\zw}[1]{\textcolor{red}{\textsf{ZW: #1}}}
\newcommand{\balpha}{\pmb{\alpha}}
\newcommand{\bc}{\pmb{c}}
\usepackage{xifthen}
\newcommand{\ctrldX}[1][]{\ifthenelse{\isempty{#1}}{X^{\balpha,\bc}}{X^{#1}}}
\newcommand{\infgen}[1][]{\ifthenelse{\isempty{#1}}{\mathcal{L}^{\alpha,c}}{\mathcal{L}^{#1}}}
\DeclareMathOperator*{\SumInt}{%
\mathchoice%
  {\ooalign{$\displaystyle\sum$\cr\hidewidth$\displaystyle\int$\hidewidth\cr}}
  {\ooalign{\raisebox{.14\height}{\scalebox{.7}{$\textstyle\sum$}}\cr\hidewidth$\textstyle\int$\hidewidth\cr}}
  {\ooalign{\raisebox{.2\height}{\scalebox{.6}{$\scriptstyle\sum$}}\cr$\scriptstyle\int$\cr}}
  {\ooalign{\raisebox{.2\height}{\scalebox{.6}{$\scriptstyle\sum$}}\cr$\scriptstyle\int$\cr}}
}
\hypersetup{
    colorlinks=true,
    linkcolor=blue,
    filecolor=magenta,      
    urlcolor=blue,
    }
\everymath=\expandafter{\the\everymath\displaystyle}

\newtheorem{theorem}{Theorem}
\newtheorem{remark}{Remark}

\newtheorem{lemma}{Lemma}

\newtheorem{proposition}{Proposition}
\newtheorem{definition}{Definition}

\newtheorem{example}{Example}
\DeclareMathSizes{24}{24}{24}{24}
\DeclareMathSizes{24}{24}{24}{24}
\title{
Inference of  Utilities and Time Preference in Sequential Decision-Making
}
\date{\today}

\author{
Haoyang Cao
\thanks{ Department of Applied Mathematics and Statistics, Johns Hopkins University. \textbf{Email:} hycao@jhu.edu.  H. C. is partially supported by the departmental startup fund.}
\and
Zhengqi Wu \thanks{Epstein Department of Industrial and Systems Engineering, University of Southern California. \textbf{Email:} \{zhengqiw,renyuanx\}@usc.edu. R. X. is partially supported by the NSF CAREER award DMS-2339240 and a JP Morgan Faculty Research Award.}
\and
Renyuan Xu $^\dagger$}

\begin{document}
\maketitle

\begin{abstract}
This paper introduces a novel stochastic control framework to enhance the capabilities of automated investment managers, or robo-advisors, by accurately inferring clients' investment preferences from past activities. Our approach leverages a continuous-time model that incorporates utility functions and a generic discounting scheme of a time-varying rate, tailored to each client's risk tolerance, valuation of daily consumption, and significant life goals.  We address the resulting time inconsistency issue through state augmentation and the establishment of the dynamic programming principle and the verification theorem. Additionally, we provide sufficient conditions for the identifiability of client investment preferences. To complement our theoretical developments, we propose a learning algorithm based on maximum likelihood estimation within a discrete-time Markov Decision Process framework, augmented with entropy regularization. We prove that the log-likelihood function is locally concave, facilitating the fast convergence of our proposed algorithm. Practical effectiveness and efficiency are showcased through two numerical examples, including Merton's problem and an investment problem with unhedgeable risks.
 
 Our proposed framework not only advances financial technology by improving personalized investment advice but also contributes broadly to other fields such as healthcare, economics, and artificial intelligence, where understanding individual preferences is crucial.
\end{abstract}

\section{Introduction}

Automated investment managers, commonly known as robo-advisors, have emerged as a modern alternative to traditional financial advisors in recent years \cite{capponi2022personalized,d2019promises,rossi2020benefits}. The effectiveness and viability of robo-advisors depend significantly on their ability to provide customized financial guidance tailored to the unique needs of each client. To provide impactful personalized advice, two critical steps must be undertaken:  first, accurately estimate the client’s investment preferences, and second, formulate investment recommendations that align with these preferences. This paper focuses on the first step, involving a detailed analysis of the client's investment preferences. 

More often than not, it is difficult for the automated investment manager to have full access to clients' investment preferences. Therefore, it is worth exploring whether it is possible to infer relevant information by observing the clients' past investment activities. On the other hand, inferring a client's investment preferences is typically challenging, as it involves several complex aspects that vary from individual to individual. For example, clients may have short-term or long-term investment objectives \cite{jin2008behavioral}. Additionally, they might exhibit varying utility functions \cite{musiela2006investments,zariphopoulou2001solution}, reflecting distinct risk tolerance related to profit-and-loss (PnL) outcomes and valuation of daily consumption. Furthermore, individuals often demonstrate diverse time preferences in terms of the trade-off between immediate and deferred outcomes \cite{bjork2010general}. Finally, clients may have specific life goals \cite{capponi2024continuous}, such as saving for their children's education or building a retirement nest egg, rather than focusing solely on generating the highest possible portfolio return or beating the market.

The inference of {preferences} in sequential decision-making is a critical component not only for financial investments but also in other fields, leveraging insights into individual behaviors to optimize decisions and predict outcomes. In economics, utility functions are inferred to model consumer behavior, guiding businesses in product development and pricing strategies \cite{derbaix1985consumer,shin2021targeted}.  Healthcare professionals use inferred utility functions to evaluate patient preferences regarding different treatment options, which is essential for effective healthcare management and policy-making \cite{chewning2012patient,richesson2010patient}. Additionally, in artificial intelligence, particularly in areas like reinforcement learning (RL) and game theory, inferring utility functions helps in designing algorithms that can predict and mimic human decision-making processes, enhancing the interaction between humans and machines \cite{boularias2011relative,christiano2017deep}.

\paragraph{Our framework, results, and contributions.} 
We propose a novel stochastic control framework in continuous time that incorporates all the aforementioned investment preferences. This framework includes two utility functions that allow the client to define their risk tolerance related to the PnL outcomes and their valuation of daily consumption. Additionally, it allows for a generic discounting scheme under a time-varying rate, enabling the clients to balance immediate and deferred outcomes. This time-varying discounting scheme further incorporates specific life goals by assigning greater importance to times of significant expenditures, such as college tuition fees for children. Lastly, we address control problems on both finite-time and infinite-time horizons to accommodate clients' preferred investment duration. The control problem is time inconsistent under the generic time-varying discounting scheme. We address this issue by state augmentation to account for the cumulative discount rate. We study the well-definedness of the augmented control framework by establishing the regularity of the value function, the dynamic programming principle (DPP), and the verification theorem (see Propositions \ref{prop:gen-dpp}, \ref{prop:gen-vrf}, \ref{prop:gen-viscosity} for finite-time horizon and Propositions \ref{prop:gen-dpp-inf}, \ref{prop:gen-vrf-inf}, \ref{prop:gen-viscosity-inf} for infinite-time horizon). In addition, we identify sufficient conditions for identifying both the utility functions and the discounting scheme by {\it solely} observing the optimal policies provided by the client (see Theorem \ref{thm:gen-idtf} for finite-time horizon and Theorem \ref{thm:gen-idtf-inf} for infinite-time horizon). 

To complement the above theoretical framework, we propose an inference procedure based on maximum likelihood estimation. To demonstrate the effectiveness of this procedure, along with tractable theoretical guarantees, we focus on a specific case within the discrete-time Markov Decision Process (MDP), featuring Shannon's entropy regularization over an infinite-time horizon. The discrete-time MDP is especially relevant in the context of statistical inference and machine learning. The entropy term encourages full exploration of the state-action space and simultaneously introduces smoothness into the analysis \cite{haarnoja2017reinforcement}. We employ a parametric framework where the client uses an exponential discounting scheme, parameterized by $\bar{\rho}$, and a utility function parameterized by $\bar{\theta}$. Both sets of parameters are unknown to the automated investment manager. Mathematically, we show that the true preference parameter $(\bar{\rho},\bar{\theta})$ is a stationary point of the log-likelihood function and the log-likelihood function is locally concave near $(\bar{\rho},\bar{\theta})$; see Proposition \ref{lemma:gradient_L} and Theorem \ref{lemma:landscape}. This landscape property facilitates the design of a gradient-based algorithm to update the inferred preference parameter. We demonstrate the promising performance of our algorithm through two examples--Merton's problem and an investment problem under unhedgeable risks.

Considering the wide-ranging applications and the versatility of our proposed framework, we use the term ``inference agent'' instead of ``automated investment manager'' to describe the individual who interacts with the clients and infers their preferences.


\paragraph{Related literature and comparisons to our results.} Our developments are associated with several lines of literature as follows.\\

\noindent \underline{Utility inference.}  
 Back in 1964, Kalman \cite{Kalman64} asked the question of whether it is possible to recover the quadratic cost by observing an optimal linear policy; a similar question was also considered by \citet{boyd1994linear}. In fact, economists have long been interested in such questions within the context of determining utility functions from observations { such as \citet{samuelson1948consumption} and \citet{rickter1966revealed}}. For instance, \citet{keeney_raiffa_1993} studied the proper rank of actions based on some deterministic evaluations under a static setting. \citet{sargent1978estimation} later extended this question into a dynamic setting where the actions were specified as labor demand and evaluations as wages. {\citet{dybvig1997recovery} paid special attention to the {\it recoverability} or {\it identifiability} of utility and showed that Von Neumann-Morgenstern preferences over terminal consumption can be inferred from wealth process of a discrete-time, binomial model or continuous-time Gaussian model.}
 
{\citet{cox2014utility} studied the utility inference problem for the optimal consumption and allocation of wealth in continuous time by observing the actions of the client. The authors observed that there are infinitely many utility functions generating a given consumption pattern when the dynamic is deterministic and the consumption and investment strategies have to satisfy a consistency condition in the stochastic setting. 
\citet{el2021recover} took a ``forward-looking'' perspective of the connection between the observable process \(\{\mathscr{X}_t(x)>0:\mathscr{X}_0(x)=x>0\}\) (i.e., the characteristic process) and the corresponding utility process \(\{U(t,z):z>0,U(0,z)=u(z)\}\) (i.e., the dynamic utility); see the concept forward utility proposed by \citet{Musiela2007}. Different than the backward-looking perspective where the connection between the observable and utility is governed by some Markovian decision-making rule, the authors interpreted such a connection as the martingale property of the process \(\{U(t,\mathscr{X}_t(x)\}\), since Markov property no longer existed under the forward-looking viewpoint. To fully explore the concavity of utilities, the authors introduced an adjoint process of \(\mathscr X\), \(\{Y_t(y):Y_0(y)=y\}\), representing the decreasing marginal utility \(\{U_z(t,z)\}\) so that \(\{Y_t(u_z(x))=U_z(t,\mathscr X_t(x))\}\). Given the initial utility \(u\), the observable process \(\mathscr X\) and its adjoint process \(Y\), the authors fully characterized the martingale dynamic utility and its dual form via the It\^o-Ventzel formula and showed that they are solutions to some Hamilton-Jacobi-Bellman-type stochastic partial differential equations; this set of analytical tools was introduced in \cite{EKM3013an} and \cite{el2019construction}. In \cite{el2024bi}, the authors also extended the result of \cite{el2021recover} to allow an exogenous default time \(\tau\). 
 }
 
 In recent years, utility inference has been integrated with machine learning to embrace the potential of the big data era (and the progress is summarized in the next paragraph). In addition, inference problems in sequential decision-making for modern applications are more complex than inferring solely the utility function. Other preferences such as time preferences and specific investment goals should also be included, leading to the main formulation of our paper.\\

\noindent {\underline{Theory of inverse optimal control.} Inverse optimal control aims at inferring the underlying reward function that motivates the observed behavior of a rational agent in a sequential decision-making framework; within the context of MDP, inverse optimal control is also known as inverse reinforcement learning (IRL). In this area,  \citet{ng2000algorithms} considered a particular setting that the true reward function is some linear combination of several action-free basis functions and that the true reward function maximally distinguishes the observed policy from the rest. They reformulated this question into a constrained linear programming problem eventually leading to a well-defined solution. In \cite{abbeel2004apprenticeship}, the reward was assumed to be a linear combination of several features that best distinguish the demonstrated policy from other policies. The key assumption in both works is that the true reward function should maximize the margin between observations and the other policies. It also played a central role in the model of the well-known GAIL (generative adversarial imitation learning) algorithm \cite{ho2016generative}. Other than the ``maximum margin'' setting, another commonly adopted setting in IRL is to assume that an observed randomized policy should maximize the causal entropy of an underlying regularized MDP. For instance, \citet{ziebart2008maximum} studied the maximum entropy IRL based on known features. They assumed that the reward is a linear function of such features. \citet{ziebart2010modeling}
extended this approach to a selected set of non-linear rewards; see also \cite{levine2011nonlinear} and \cite{boularias2011relative} for similar settings. \citet{wulfmeier2015maximum} followed this approach but with rewards represented by neural networks. \citet{finn2016connection} combined the idea of adversarial training and IRL. They trained a discriminator to recover the reward function. \citet{reddy2019sqil} proposed a soft Q imitation learning algorithm to imitate the expert's policy by learning her Q function. \citet{garg2021iq} proposed an algorithm to learn the soft Q function which implicitly represents both the reward function and the policy. \citet{zeng2022maximum} adopted the maximum likelihood estimator and showed that their algorithm converges to a stationary point under a finite-time guarantee.  Back to our preference inference problem, since it is to infer the utility functions and the time preferences of the client simultaneously, these existing IRL algorithms are {\it not} directly applicable. Such a {\it multi-facet} inference problem motivates our main algorithm. Furthermore, 
we are able to provide a loss landscape analysis that facilitates fast convergence of our proposed algorithm; see Proposition \ref{lemma:gradient_L} and Theorem \ref{lemma:landscape}.  \\

\noindent\underline{Identifiability issues in IRL.} In 1998, Russell \cite{russell1998learning} pointed out the ill-posedness of inverse optimal control or IRL problems under a generic setting. Both the ``maximum margin'' and the ``maximum entropy'' settings mentioned above are reasonable assumptions to ameliorate this ill-posedness. Nonetheless, without prior access to the underlying true reward function, it is difficult to verify either one of them. To guarantee identifiability in IRL, alternative and more verifiable conditions and assumptions are required. Under an entropy regularized MDP setting, \citet{cao2021identifiability} pointed out two possible remedies for the identifiability issue. One way is to provide additional observations of the same agent (i.e., keeping the underlying reward function the same) under different environments; see also a repeated IRL setting proposed in \cite{amin2016towards} and \cite{amin2017repeated}. It was shown in \cite{cao2021identifiability} that under proper technical conditions on the transition kernels, observations from two distinct environments would suffice. Another approach is to provide additional structural assumptions on the MDP environment or the family of candidate reward functions based on prior domain knowledge; see also the identification of an action-free reward in \cite{fu2017learning}. Both \citet{cao2021identifiability} and \citet{Kim2021} provided sufficient structural conditions for the MDP environment that guarantee identifiability. 

However, as pointed out by \citet{schlaginhaufen2023identifiability}, the identifiability may no longer hold without the entropy regularization. In addition, the majority of these previous studies rely on the {\it full disclosure of the MDP environment}, including the transition kernel, time horizon, and the rate of an {\it exponential discounting scheme}. Though \citet{dong2024towards} provided a mathematical formulation and an algorithm for the partial information setting, it remains to be explored whether identifiability of both the unknown MDP information and the true reward function is viable. In this paper, we establish such identifiability for our preference inference problem, which is also one of the major theoretical contributions; see Theorems \ref{thm:gen-idtf} and \ref{thm:gen-idtf-inf}. \\

\noindent \underline{Time inconsistency in stochastic control.} Unlike assuming an exponential discounting scheme for the client, a general discounting scheme may lead to a {\it time-inconsistent} policy. In economics, one of the earliest studies on the inconsistency in dynamic utility maximization is \cite{strotz1955myopia}, where the optimality of the problem derived today is different from that of tomorrow due to some non-exponential discounting mechanism. Later \citet{pollack1968consistent} proposed a game-theoretic consistent planning approach for the discrete-time problem, where the game is among decision makers at different time steps and the optimal decision path is considered to be the Nash equilibrium. There has been a line of works following this consistent planning approach under both discrete- and continuous-time settings; see, for instance, \cite{bjork2014theory,bjork2017time,ekeland2010golden,hu2012time,hu2017time,yong2012time}, and more recently, \cite{dai2023learning,hernandez2023me}. Apart from this game-theoretic approach, \citet{karnam2017dynamic} introduced the idea of ``dynamic utility'' to a family of time-inconsistent optimization problems over a {\it finite-time horizon}. By modeling the utility as the solution to a backward stochastic differential equation (BSDE), the DPP could be revived. {For an {\it infinite-time horizon} setting which is suitable to model a long-run investment planning problem though, this BSDE approach can no longer be applied. Hence we propose a different way to revive DPP; see Propositions \ref{prop:gen-dpp} and \ref{prop:gen-dpp-inf} in Section \ref{sec:continuous_time}.} \\

\noindent \underline{Robo-advising.} Robo-advising has emerged over the last two decades as an alternative to traditional human financial advising, addressing limitations such as the human advisors' limited knowledge and high service fees \cite{capponi2022personalized,d2021robo,d2019promises}. Here, we mainly review some papers that explore the machine learning and inference aspects of this subject. The first RL algorithm for a robo-advisor was proposed by \citet{alsabah2021robo}, where the authors designed an exploration-exploitation algorithm to learn a constant risk appetite parameter and then applied a follow-the-leader type of algorithm to invest. \citet{wang2021robo} introduced a framework consisting of two agents: the first, an inverse portfolio optimization agent, infers a risk preference parameter and the corresponding expected return; the second aggregates the learned information to formulate a new multi-period portfolio optimization problem solved by deep learning. To transcend the rather single-facet inference settings above, the theoretical framework and the numerical procedure in our paper are designed to capture the multiple investment needs of a client. 

\section{Continuous-time Framework}
\label{sec:continuous_time}
In this section, we study a continuous-time framework of the joint consumption-allocation problem of an investing client. Her wealth consists of a risk-free asset and a risky asset. What distinguishes this framework from the classical ones is that the client holds a {\it general preference of time}, that is, the discounting scheme is {\it not necessarily exponential}. This could possibly lead to time-inconsistent decision-making. First, for the {\it optimal control} problem, we analyze the time-inconsistent dynamical decision-making problem for such a client, assuming the client's utility functions of consumption and wealth as well as her time preference are fully disclosed. The optimal decision relies on reviving a suitable DPP under this framework. Then, for the {\it inverse optimal control} problem, we establish an identifiability result for both the utility functions and the time preference of the client, assuming instead her optimal joint consumption-allocation plan is disclosed. Such an identifiability result provides inspirations for the algorithm proposed in Section \ref{sec:discrete-time}. 

\subsection{Finite-time Horizon}\label{subsec:gen-finite}
 We first focus on a finite-time horizon setting, with a decision horizon $T$, to address scenarios where the client has a short-term investment plan.
\paragraph{Market dynamics and client's wealth.}  Let $(\Omega, \mathcal{F}, \mathbb{F} = (\mathcal{F}_t)_{t\ge 0},\mathbb{P})$ be a filtered probability space, supporting a one-dimensional $\mathbb{F}$-Brownian motion $W$. Assume there is a bond and a stock in the investment universe. The price of the bond follows
\begin{eqnarray}\label{eq:bond_price}
    {\rm d} S_t^0 = r {\rm d}  t,
\end{eqnarray}
and the price of the stock follows
\begin{eqnarray}\label{eq:asset_price}
    {\rm d }S_t = S_t(\mu {\rm d}t + \sigma {\rm d}W_t).
\end{eqnarray}
Assume the client choose an allocation process \(\balpha=\{\alpha_t\}_{t\in[0,T]}\)  and a consumption process \(\bc=\{c_t\}_{t\in[0,T]}\) with \(c_t\geq0\).  Namely, the client allocates $\alpha_t$ proportion of wealth to the stock and $1-\alpha_t$ proportion of wealth to the bond at time $t$. In addition,  the client is also making consumption $c_t$ to achieve certain satisfaction in life.

Fixing a sufficiently large constant \(M\in\R^+\) and introducing a compact space \(\mathcal{K}=[-M,M]\times[0,M]\), define
\begin{equation}
    \label{eq:adm-set}
    \begin{aligned}
    \calA\coloneqq&\left\{(\balpha,\bc)\{(\alpha_t,c_t)\}_{t\geq0}\,\biggl|\,(\alpha_t,c_t)\in\mathcal{K},\, (\alpha_t,c_t)\in\bar\calF_t\coloneqq\sigma\left(\sigma(\beta_s,0\leq s\leq t)\times\mathcal{F}^W_t\right),\right.\\
    &\hspace{100pt}{}\E\left[|\ctrldX_t|^2\right]<\infty,\,\forall t\geq0\biggl\}
    \end{aligned}
\end{equation}
as the admissible set of all possible joint consumption-allocation processes. Hence the wealth process follows:
\begin{eqnarray}
   d\ctrldX_t=\left\{\ctrldX_t\left[\alpha_t\mu+(1-\alpha_t)r\right]-c_t\right\}dt+\sigma\alpha_t\ctrldX_tdW_t.\label{eq:gen-wealth}
\end{eqnarray}

 \paragraph{Client's preference.}   
 In the finite-time horizon, the preference of the client can be characterized by a pair of utility functions and a discount scheme.
 More specifically, consider utility functions $U_1, U_2$ that belong to the following class
\begin{equation}
    \label{eq:gen-util}
    \begin{aligned}
    \mathcal U\coloneqq\biggl\{U:\R\to[-\infty,+\infty)\biggl|\,\,&U\text{ is strictly positive, increasing and concave on }(0,+\infty),\\
    &\text{there exists a sufficiently large constant }C\in\R^+\text{ such that}\\
    &|U(x)|\leq C\left(1+x^2\right)\text{ for all }x\in(0,\infty),
    \\&
    U\in\calC^2((0,+\infty)),\,U(x)=U(0)\text{ for }x\leq 0\biggl\}.
\end{aligned}
\end{equation}
Here $U_1$ quantifies the the client's evaluation regarding the consumption whereas $U_2$ quantifies her evaluation regarding the terminal wealth at the end of the investment plan.

\paragraph{General discounting scheme.} We are particularly interested in a client that is subject to a general discounting scheme \(\beta=\{\beta_t=\beta(t)\}_{t\geq0}\), where
\begin{itemize}
    \item \(\beta_t\in[0,1]\) for all \(t\in[0,T]\); and
    \item there exists \(\dot{\beta}:[0,\infty)\to\R\) such that \(\dot{\beta}\) is bounded and integrable on \([0,t]\) with \(\beta_t=\int_0^t\dot{\beta}_sds+\beta_0\) for any \(t>0\).
\end{itemize}
{Such a discounting scheme \(\{\beta_t\}_{t\geq0}\) reflects a generic {\it time preference} of the client. A time-varying discounting rate could account for different levels of appreciation for the immediate outcome and the delayed fulfillment. It could also provide the flexibility of assigning greater importance to times of significant expenditures, such as college tuition for children and down-payment of a house.} 

Then for any \((t,x,z)\in[0,T]\times\R\times[0,1]\), define the total reward as 
\begin{equation}
    \label{eq:acc-util}
    J(t,x,z,\balpha,\bc)\coloneqq\E\left[\int_t^T\beta_sU_1(c_s)ds+\beta_TU_2(X_T)\,\biggl|\,X_t=x,\beta_t=z\right]
\end{equation}
subject to the wealth process \eqref{eq:gen-wealth} and 
\begin{align}
    \dd \beta_t=\dot{\beta}_t \dd t.\label{eq:gen-discnt}
\end{align}
For any \((x,z)\in\R\times[0,1]\), define the value function as follows,
\begin{equation}
    \label{eq:gen-value}
V(t,x,z)=\sup_{(\balpha,\bc)\in\calA}J(t,x,z,\balpha,\bc),\,t\in[0,t);\ \ V(T,x,z)=zU_2(x).
\end{equation}
subject to \eqref{eq:gen-wealth} and \eqref{eq:gen-discnt}. 

In this section, we also aim to recover the DPP to the above time-inconsistent utility optimization problem \eqref{eq:gen-value}, where the time-inconsistency is particularly due to the general discounting scheme. We take a different approach than the BSDE characterization of dynamic utility in \cite{karnam2017dynamic}; instead, we extend the state space to incorporate the discounting process (similar to \cite{bauerle2014more}) and then re-establish DPP accordingly. 

\subsubsection{Preliminary Analysis}
First, we establish the well-definedness of the control problem \eqref{eq:gen-wealth}-\eqref{eq:gen-discnt} and introduce some analytical properties associated with it.

\begin{lemma}
    \label{lem:pos-wealth}
    Assume that \(U_1,U_2\in\mathcal{U}\). Moreover, assume that \(U_1(0)=0\) and \(U_2(0)=-\infty\). For any \((t,x,z)\in[0,T]\times(0,\infty)\times[0,1]\), if the policy $\balpha^*,\bc^*$ satisfies that \(J(t,x,z,\balpha^*,\bc^*)=V(t,x,z)\), then it holds almost surely that 
\begin{eqnarray}
    \ctrldX[\balpha^*,\bc^*]_s\in(0,\infty)\quad \text{for all}\quad s\in[t,T],
\end{eqnarray}
  where \(\ctrldX[\balpha^*,\bc^*]\) solves \eqref{eq:gen-wealth} on \([t,T]\) given \((\balpha,\bc)=(\balpha^*,\bc^*)\) and \(\ctrldX_t=x\).
\end{lemma}
\begin{proof}
    For any \((\balpha,\bc)\in\calA\), we have \(\ctrldX_s\leq \ctrldX[\balpha,\pmb{0}]_s\) for all \(s\in[t,T]\) almost surely. Notice that \(\ctrldX[\balpha,\pmb{0}]_{t'}=\ctrldX[\balpha,\pmb{0}]_t\exp\left\{\int_t^{t'}\alpha_l(\mu-r)+r-\frac{\sigma^2\alpha_l^2}{2}dl+\int_t^{t'}\sigma\alpha_ldW_l\right\}\) for \({t'}\geq t\). If \(\ctrldX_{t'}\leq0\), then \[\ctrldX_T\leq\ctrldX[\balpha,\pmb{0}]_T\leq0,\] and hence \(J(t,x,z,\balpha,\bc)=-\infty\). On the other hand,
    \[J(t,x,z,\balpha,\pmb{0})=\beta_T\E{U_2\left(x\exp\left\{\int_t^{T}\alpha_l(\mu-r)+r-\frac{\sigma^2\alpha_l^2}{2}dl+\int_t^{T}\sigma\alpha_ldW_l\right\}\right)}>0.\]
    Then if \(J(t,x,z,\balpha^*,\bc^*)=V(t,x,z)\), then \(\ctrldX[\balpha^*,\bc^*]_s\in(0,\infty)\) for all \(s\in[t,T]\) almost surely.
\end{proof}

\begin{lemma}
    \label{lem:gen-v-1}
    Assume that \(U_1,U_2\in\mathcal{U}\). Moreover, assume that \(U_1(0)=0\) and \(U_2(0)=-\infty\). Then it holds that the value function \(V:[0,T]\times\R\times[0,1]\to[-\infty,+\infty)\) defined in \eqref{eq:gen-value} is strictly concave and strictly increasing in \(x\in(0,\infty)\) given any \((t,z)\in[0,T]\times(0,1]\).
\end{lemma}
\begin{proof}
Fix any \((t,x,z)\in[0,T]\times(0,\infty)\times(0,1]\). 
\begin{enumerate}
    \item {\it Strictly concave and positive}. Take \(y\in(0,\infty)\setminus\{x\}\) and \(\lambda\in(0,1)\). Define \(u=\lambda x+(1-\lambda)y\) and \(u\in(0,\infty)\). Take any \((\balpha^x,\bc^x),(\balpha^y,\bc^y)\in\calA\) and define \((\balpha^u,\bc^u)=\lambda(\balpha^x,\bc^x)+(1-\lambda)(\balpha^y,\bc^y)\). Then it immediately follows that \((\balpha^u,\bc^u)\in \calA\). Let \(\ctrldX[\balpha^u,\bc^u]\) (resp. \(\ctrldX[\balpha^x,\bc^x]\) or \(\ctrldX[\balpha^y,\bc^y]\)) be the solution to the SDE \eqref{eq:gen-wealth} over \([t,T]\) given \((\balpha,\bc)=(\balpha^u,\bc^u)\) (resp. \((\balpha,\bc)=(\balpha^x,\bc^x)\) or \((\balpha,\bc)=(\balpha^y,\bc^y)\)) and \(\ctrldX_t=u\) (resp. \(\ctrldX_t=x\) or \(\ctrldX_t=y\)). Then we have \[\ctrldX[\balpha^u,\bc^u]_s = \lambda\ctrldX[\balpha^x,\bc^x]_s+(1-\lambda)\ctrldX[\balpha^y,\bc^y]_s,\ \ s\in[t,T].\] 
    By Lemma \ref{lem:pos-wealth}, we can assume that both \(\ctrldX[\balpha^x,\bc^x]_s\) and \(\ctrldX[\balpha^y,\bc^y]_s\) are strictly positive for \(s\in[t,T]\) almost surely. Since \(U_1,U_2\in\mathcal{U}\), then 
    \[J(t,u,z,\balpha^u,\bc^u)>\lambda J(t,x,z,\balpha^x,\bc^x)+(1-\lambda)J(t,y,z,\balpha^y,\bc^y)>0.\]
    Taking the supremum over both \((\balpha^x,\bc^x)\) and \((\balpha^y,\bc^y)\), 
    \[V(t,u,z)>\lambda V(t,x,z)+(1-\lambda)V(t,y,z)>0.\]
    \item {\it Strictly increasing}. Fix any \(\Delta x>0\) take any \((\balpha,\bc)\in\calA\) such that \(\ctrldX_s>0\) for \(s\in[t,T]\) almost surely. Let \(\widehat{\ctrldX}\) be the solution to \eqref{eq:gen-wealth} given \(\ctrldX_t=x+\Delta x\). Then,
    \[\Delta_T :=\widehat{\ctrldX_T}-\ctrldX_T=\Delta x\exp\left\{\int_t^{T}\alpha_l(\mu-r)+r-\frac{\sigma^2\alpha_l^2}{2}dl+\int_t^{T}\sigma\alpha_ldW_l\right\}>0\text{ a.s.},\]
    and therefore
    \[J(t,x+\Delta x,z,\balpha,\bc)-J(t,x,z,\balpha,\bc)=\beta_T\E\left[U_2\left(\widehat{\ctrldX_T}\right)-U_2(\ctrldX_T)\right]>0.\]
    Hence, \(V(t,x+\Delta x,z)>V(t,x,z)\).
\end{enumerate}
\end{proof}

Having established some preliminary properties of the value function, we first show a necessary condition for the value function \eqref{eq:gen-value}. 
\begin{proposition}[Dynamic programming principle (DPP)]\label{prop:gen-dpp}
Take the same assumptions on \(U_1,U_2\) as in Lemma \ref{lem:pos-wealth}. For any \((t,x,z)\in[0,T)\times\R\times[0,1]\) and \(\tau\in\mathbb T_t\) where \(\mathbb T_t\) denotes all \(\{\bar\calF_t\}_{t\geq0}\)-adapted stopping times \(\tau\) such that \(\tau\in[t,T]\) a.s., the value function \(V\) defined in \eqref{eq:gen-value} satisfies
\begin{equation}
    \label{eq:gen-dpp}
    \tag{DPP}
    V(t,x,z)=\sup_{(\balpha,\bc)\in\calA}\E\left[\int_t^{\tau}\beta_sU_1(c_s)ds+V(\tau, \ctrldX_{\tau},\beta_{\tau})\biggl|\ctrldX_t=x,\beta_t=z\right],
\end{equation}
with \(V(T,x,z)=zU_2(x)\).
\end{proposition}
\begin{proof}
    Fix any \((t,x,z)\in[0,T)\times\R\times[0,1]\), \((\balpha,\bc)\in\calA\) and \(\tau\in\mathbb T_t\). We have 
    \begin{align*}
    J(t,x,z,\balpha,\bc)&= \E\left[\int_t^{\tau}\beta_sU_1(c_s)ds+\int_{\tau}^{T}\beta_sU_1(c_s)ds+\beta_TU_2(\ctrldX_T)\biggl|\ctrldX_t=x,\beta_t=z\right]\\
    &=\E\left[\int_t^{\tau}\beta_sU_1(c_s)ds\biggl|\ctrldX_t=x,\beta_t=z\right]\\
    &\hspace{20pt}+\E\left[\E\left[\int_{\tau}^{T}\beta_sU_1(c_s)ds+\beta_TU_2(\ctrldX_T)\biggl|\ctrldX_{\tau},\beta_{\tau}\right]\biggl|\ctrldX_t=x,\beta_t=z\right]\\
    &=\E\left[\int_t^{\tau}\beta_sU_1(c_s)ds\biggl|\ctrldX_t=x,\beta_t=z\right]\\
    &\hspace{20pt}+\E\left[\E\left[J(\tau,X,Z,\balpha,\bc)\biggl|\tau, X=\ctrldX_{\tau},Z=\beta_{\tau}\right]\biggl|\ctrldX_t=x,\beta_t=z\right].
\end{align*}
By the definition given by \eqref{eq:gen-value}, for any \(\epsilon>0\) and \(\Delta t\in[0,T-t]\), there exists \((\balpha^{\epsilon,t+\Delta t,x,z},\bc^{\epsilon,t+\Delta t,x,z})\in\calA\) such that
\begin{equation}\label{eq:gen-upper}
    J(t+\Delta t,x,z,\balpha^{\epsilon,t+\Delta t,x,z},\bc^{\epsilon,t+\Delta t,x,z})>\sup_{(\balpha,\bc)\in\calA}J(t+\Delta t,x,z,\balpha,\bc)-\epsilon=V(t+\Delta t, x,z)-\epsilon,
\end{equation}
and
\begin{equation}\label{eq:gen-lower}
    J(t+\Delta t, x,z,\balpha,\bc)\leq J(t+\Delta t,x,z,\balpha^{\epsilon,t+\Delta t,x,z},\bc^{\epsilon,t+\Delta t,x,z})\leq V(t+\Delta t, x,z).
\end{equation}
Then consider \((\bar\balpha,\bar\bc)=\{(\bar\alpha_s,\bar c_s)\}_{s\in[t,T]}\) such that
\begin{equation*}
    (\bar\alpha_s,\bar{c}_s)=(\alpha_s,c_s)\mathbbm{1}\{\tau>s\}+ \left(\alpha^{\epsilon,\tau,\ctrldX_{\tau},\beta_{\tau}}_s,c^{\epsilon,\tau,\ctrldX_{\tau},\beta_{\tau}}_s\right)\mathbbm{1}\{\tau\leq s\},
\end{equation*}
where 
\begin{equation*}
    \beta_{\tau}=z+\int_t^{\tau}\dot{\beta}_sds,\ \
    \ctrldX_{\tau}=x+\int_t^{\tau}d\ctrldX_s,
\end{equation*}
according to \eqref{eq:gen-wealth} and \eqref{eq:gen-discnt}. Notice that \((\bar\balpha,\bar\bc)\in\calA\). By \eqref{eq:gen-upper} and \eqref{eq:gen-lower}, for any \((\balpha,\bc)\in\calA\), we have
\[\begin{aligned}
V(t,x,z)&\geq J(t,x,z,\bar\balpha,\bar\bc)>\E\left[\int_t^{\tau}\beta_sU_1(c_s)ds\biggl|\ctrldX_t=x,\beta_t=z\right]\\
&\hspace{50pt}+\E\left[V(\tau, \ctrldX_{\tau},\beta_{\tau})\biggl|\ctrldX_t=x,\beta_t=z\right]-\epsilon
\end{aligned}\]
for any \(\epsilon>0\), and
\[J(t,x,z,\balpha,\bc)\leq \E\left[\int_t^{\tau}\beta_sU_1(c_s)ds+V(\tau, \ctrldX_{\tau},\beta_{\tau})\biggl|\ctrldX_t=x,\beta_t=z\right].\]
It follows that for any \((t,x,z)\in[0,T)\times\R\times[0,1]\) and \(\tau\in\mathbb T_t\),
\begin{equation*}
    V(t,x,z)=\sup_{(\balpha,\bc)\in\calA}\E\left[\int_t^{\tau}\beta_sU_1(c_s)ds+V(\tau, \ctrldX_{\tau},\beta_{\tau})\biggl|\ctrldX_t=x,\beta_t=z\right].
\end{equation*}
\end{proof}
For any \(\alpha\in\R\) and \(c\in\R^+\), define the following operator
\[\infgen \phi(t,x,z)=\left\{\left[\alpha(\mu-r)x\right]-c\right\}\partial_x\phi(t,x,z)+\frac{\sigma^2\alpha^2}{2}x^2\partial_x^2\phi(t,x,z),\]
for any test function \(\phi\in\calC^{\infty}_b([0,T)\times\R\times\R^+)\bigcap\calC_b^0([0,T)\times\R\times\R^+)\). Following the DPP under a generic discounting scheme \eqref{eq:gen-dpp}, we have the following Hamilton-Jacobi-Bellman (HJB) equation
\begin{equation}
    \label{eq:gen-hjb}
    \tag{HJB}\left\{
    \begin{aligned}
    &\partial_tV(t,x,z)+\dot{\beta}_t\partial_zV(t,x,z)+rx\partial_xV(t,x,z)+\sup_{(\alpha,c)\in\mathcal K}\left\{zU_1(c)+\infgen V(t,x,z)\right\}=0,\ \ t\in[0,T);\\
    &V(T,x,z)=zU_2(x).
    \end{aligned}\right.
\end{equation}

The next result provides sufficient conditions for the value function in \eqref{eq:gen-value} regarding classical solutions to \eqref{eq:gen-hjb}. 


\begin{proposition}
    \label{prop:gen-vrf}
    Take the same assumptions on \(U_1,U_2\) as in Lemma \ref{lem:pos-wealth}. Let \(w:[0,T]\times\R\times[0,1]\to\R\) be a function such that 
    \[w\in\calC^{1,2,1}\left([0,T]\times\R\times[0,1]\right)\bigcap\calC^0\left([0,T]\times\R\times[0,1]\right),\]
    and there exists a constant \(C>0\) with
    \[w(t,x,z)\leq C(1+|x|^2),\ \ \forall (t,x,z)\in[0,T]\times\R\times[0,1].\]
    \begin{enumerate}
        \item Assume that for any \((\alpha,c)\in\mathcal K\), 
        \[
        \left\{
        \begin{aligned}
            &\partial_tw(t,x,z)+\dot{\beta}_t\partial_zw(t,x,z)+rx\partial_xw(t,x,z)+zU_1(c)+\infgen w(t,x,z)\leq0,\\
            &\hspace{180pt}\forall (t,x,z)\in[0,T)\times\R\times[0,1];\\
            &w(T,x,z)\geq zU_2(x), \ \  \forall (x,z)\in\R\times[0,1].
        \end{aligned}\right.
        \]
        Then \(w\geq V\) on \([0,T]\times\R\times[0,1]\).
        \item Assume further that there exists \(\hat{\alpha}:[0,T]\times\R\times[0,1]\to[-M,M]\) and \(\hat{c}:[0,T]\times\R\times[0,1]\to[0,M]\) such that 
        \[
        \left\{
        \begin{aligned}
            &\partial_tw(t,x,z)+\dot{\beta}_t\partial_zw(t,x,z)+rx\partial_xw(t,x,z)+zU_1(\hat c(t,x,z))+\infgen[\hat{\alpha}(t,x,z),\hat{c}(t,x,z)]w(t,x,z)=0,\\
            &\hspace{180pt}\forall (t,x,z)\in[0,T)\times\R\times[0,1];\\
            &w(T,x,z)=zU_2(x), \ \  \forall (x,z)\in\R\times[0,1],
        \end{aligned}\right.
        \]
        also, with \(\beta_t=\beta_0+\int_0^t\dot{\beta}_sds\in[0,1]\) for all \(t\in[0,T]\), the following SDE,
        \[dX_t=
        \left\{X_t\left[\hat{\alpha}(t,X_t,\beta_t)(\mu-r)+r\right]-\hat{c}(t,X_t,\beta_t)\right\}dt
        +\sigma\hat{\alpha}(t,X_t,\beta_t)X_tdW_t, 
        \]
        admits a unique solution \(\ctrldX[\hat{\balpha},\hat{\bc}]\) given \(X_0=x\) for any \(x\in\R\), and 
        \[\left(\hat{\balpha}=\left\{\hat{\alpha}_t\right\}_{t\in[0,T]}=\left\{\hat{\alpha}(t,\ctrldX[\hat{\balpha},\hat{\bc}]_t,\beta_t)\right\}_{t\in[0,T]}, \hat{\bc}=\left\{\hat{c}_t\right\}_{t\in[0,T]}=\left\{\hat{c}(t,\ctrldX[\hat{\balpha},\hat{\bc}]_t,\beta_t)\right\}_{t\in[0,T]}\right)\in\calA.\]
        Then \(w= V\) on \([0,T]\times\R\times[0,1]\), with \((\hat{\balpha},\hat{\bc})\) being an optimal joint allocation-consumption process.
    \end{enumerate}
\end{proposition}
\begin{proof}
The assumptions on \(U_1\) and \(U_2\) guarantee a quadratic growth rate in \(x\). Consider arbitrary \((t,x,z)\in[0,T)\times\R\times[0,1]\) and \((\balpha,\bc)\in\calA\). 
    \begin{enumerate}
        \item Define
        \[\tau_n\coloneqq\inf\left\{s\geq t\biggl|\int_t^s|\partial_xw(u,\ctrldX_u,\beta_u)|^2du\geq n\right\},\ \ \forall n\in\mathbb N^+.\]
        Then we have \(\lim_{n\uparrow\infty}\tau_n\overset{a.s.}{=}\infty\) and the stopped process \(\left\{\int_t^{s\wedge\tau_n}\partial_xw(u,\ctrldX_u,\beta_u)dW_u\right\}_{s\in[t,T]}\) is a martingale for all \(n\in\mathbb N^+\). The for any \(s\in[t,T]\), by It\^{o}'s formula, we have
        \begin{align*}
            w(s\wedge\tau_n,\ctrldX_{s\wedge\tau_n},\beta_{s\wedge\tau_m})&=w(t,x,z)+\int_t^{s\wedge\tau_n}\left\{\partial_tw(u,\ctrldX_u,\beta_u)+\dot{\beta}_u\partial_zw(u,\ctrldX_u,\beta_u)\right.\\
            &.{}+r\ctrldX_u\partial_xw(u,\ctrldX_u,\beta_u)+\infgen[\alpha_u,c_u] w(u,\ctrldX_u,\beta_u)\Big\}\,du\\
            &+\int_t^{s\wedge\tau_n}\partial_xw(u,\ctrldX_u,\beta_u)dW_u.
        \end{align*}
        Therefore, taking expectations on both sides we have
        \begin{align*}
            &\E\left[w(s\wedge\tau_n,\ctrldX_{s\wedge\tau_n},\beta_{s\wedge\tau_n})\biggl|\ctrldX_t=x,\beta_t=z\right]=w(t,x,z)+\mathbb E\left[\int_t^{s\wedge\tau_n}\partial_tw(u,\ctrldX_u,\beta_u)\right.\\
            &\hspace{20pt}\left.{}+\dot{\beta}_u\partial_zw(u,\ctrldX_u,\beta_u)+r\ctrldX_u\partial_xw(u,\ctrldX_u,\beta_u)+\infgen[\alpha_u,c_u] w(u,\ctrldX_u,\beta_u)du\biggl|\ctrldX_t=x,\beta_t=z\right]\\
            &\hspace{20pt}\leq w(t,x,z)-\E\left[\int_t^{s\wedge\tau_n}\beta_uU_1(c_u)du\biggl|\ctrldX_t=x,\beta_t=z\right],
        \end{align*}
        where the well-posedness of \(\E\left[\int_t^{s\wedge\tau_n}\beta_uU_1(c_u)du\biggl|\ctrldX_t=x,\beta_t=z\right]\) is guaranteed by the quadratic growth rate condition on \(U_1\) and the fact that \((\balpha,\bc)\in\calA\). The quadratic growth rate assumption on \(w\) together with \((\balpha,\bc)\in\calA\) allows us to apply dominated convergence theorem and get 
        \[\begin{aligned}
        &\E\left[\beta_TU_2(\ctrldX_T)\biggl|\ctrldX_t=x,\beta_t=z\right]\leq\E\left[w(T,\ctrldX_T,\beta_T)\biggl|\ctrldX_t=x,\beta_t=z\right]\\
        \leq&w(t,x,z)-\E\left[\int_t^T\beta_uU_1(c_u)du\biggl|\ctrldX_t=x,\beta_t=z\right]\\
        \implies&w(t,x,z)\geq\E\left[\int_t^T\beta_uU_1(c_u)du+\beta_TU_2(\ctrldX_T)\biggl|\ctrldX_t=x,\beta_t=z\right]=J(t,x,z,\balpha,\bc).\end{aligned}\]
        Hence, \(w(t,x,z)\geq V(t,x,z)\) by taking the supreme of \((\balpha,\bc)\) over \(\calA\). 
        \item Applying a similar localization-and-It\^{o} argument as in the previous part, we have that for any \(s\in[t,T]\)
        \begin{align*}
            &\E\left[w(s,\ctrldX[\hat{\balpha},\hat{\bc}]_{s},\beta_{s})\biggl|X_t=x,\beta_t=z\right]=w(t,x,z)+\mathbb E\left[\int_t^{s}\partial_tw(u,\ctrldX[\hat{\balpha},\hat{\bc}]_u,\beta_u)\right.\\
            &\hspace{20pt}\left.{}+\dot{\beta}_u\partial_zw(u,\ctrldX[\hat{\balpha},\hat{\bc}]_u,\beta_u)+r\ctrldX[\hat{\balpha},\hat{\bc}]_u\partial_xw(u,\ctrldX[\hat{\balpha},\hat{\bc}]_u,\beta_u)+\infgen[\hat{\alpha}_u,\hat c_u] w(u,\ctrldX[\hat{\balpha},\hat{\bc}]_u,\beta_u)du\biggl|\ctrldX[\hat{\balpha},\hat{\bc}]_t=x,\beta_t=z\right]\\
            &\hspace{20pt}= w(t,x,z)-\E\left[\int_t^{s}\beta_uU_1(\hat{c}_u)du\biggl|\ctrldX[\hat{\balpha},\hat{\bc}]_t=x,\beta_t=z\right].
        \end{align*}
        In particular, when \(s=T\),
        \[\begin{aligned}
        &\E\left[\beta_TU_2(\ctrldX[\hat{\balpha},\hat{\bc}]_T)\biggl|\ctrldX[\hat{\balpha},\hat{\bc}]_t=x,\beta_t=z\right]=\E\left[w(T,\ctrldX[\hat{\balpha},\hat{\bc}]_T,\beta_T)\biggl|\ctrldX[\hat{\balpha},\hat{\bc}]_t=x,\beta_t=z\right]\\
        =&w(t,x,z)-\E\left[\int_t^T\beta_uU_1(\hat{c}_u)du\biggl|\ctrldX[\hat{\balpha},\hat{\bc}]_t=x,\beta_t=z\right]\\
        \implies&w(t,x,z)=\E\left[\int_t^T\beta_uU_1(\hat{c}_u)du+\beta_TU_2(\ctrldX[\hat{\balpha},\hat{\bc}]_T)\biggl|\ctrldX[\hat{\balpha},\hat{\bc}]_t=x,\beta_t=z\right]=J(t,x,z,\hat{\balpha},\hat{\bc}).\end{aligned}\]
        Then we have \(w(t,x,z)=J(t,x,z,\hat{\balpha},\hat{\bc})\leq V(t,x,z)\). Combined with the result from the previous part, we have \(w=V\) on \([0,T]\times\R\times\R^+\), with \((\hat{\balpha},\hat{\bc})\in\calA\) being a corresponding optimal joint allocation-consumption process.
    \end{enumerate}
\end{proof}
Without assuming the existence of a classical solution to \eqref{eq:gen-hjb}, we could instead consider its viscosity solution.
\begin{definition}[Viscosity solution]\label{def:visc-soln}
Denote \(\calD=[0,T)\times(0,\infty)\times[0,1]\).
\begin{enumerate}
    \item A lower semi-continuous function \(\underline{v}:\calD\to\R\) is a viscosity subsolution to \eqref{eq:gen-hjb} if for any \((t_0,x_0,z_0)\in\calD\) and any \(\phi\in\calC^{1,2,1}(\calD)\) such that
    \[\min_{(t,x,z)\in B(t_0,x_0,z_0)}(\phi-\underline{v})(t,x,z)=(\phi-\underline{v})(t_0,x_0,z_0)=0\]
    for some neighborhood \(B(t_0,x_0,z_0)\subset\calD\),
    \begin{equation}
        \label{eq:gen-subsoln}
        -\partial_t\phi(t_0,x_0,z_0)-\dot{\beta}_t\partial_z\phi(t_0,x_0,z_0)-rx_0\partial_x\phi(t_0,x_0,z_0)-\sup_{(\alpha,c)\in\mathcal K}\Big\{z_0U_1(c)+\mathcal{L}^{\alpha,c}\phi(t_0,x_0,z_0)\Big\}\leq0.
    \end{equation}
    \item An upper semi-continuous function \(\overline{v}:\calD\to\R\) is a viscosity supersolution to \eqref{eq:gen-hjb} if for any \((t_0,x_0,z_0)\in\calD\) and any \(\psi\in\calC^{1,2,1}(\calD)\) such that 
    \[\max_{(t,x,z)\in B(t_0,x_0,z_0)}(\psi-\underline{v})(t,x,z)=(\phi-\underline{v})(t_0,x_0,z_0)=0\]
    for some neighborhood \(B(t_0,x_0,z_0)\subset\calD\),
    \begin{equation}
        \label{eq:gen-supersoln}
        -\partial_t\psi(t_0,x_0,z_0)-\dot{\beta}_t\partial_z\psi(t_0,x_0,z_0)-rx_0\partial_x\psi(t_0,x_0,z_0)-\sup_{(\alpha,c)\in\mathcal K}\Big\{z_0U_1(c)+\mathcal{L}^{\alpha,c}\psi(t_0,x_0,z_0)\Big\}\geq0.
    \end{equation}
    \item A continuous function \(v:\calD\to\R\) is a viscosity solution to \eqref{eq:gen-hjb} if it is both a viscosity subsolution and a viscosity supersolution to \eqref{eq:gen-hjb}.
\end{enumerate}
\end{definition}

\begin{proposition}
    \label{prop:gen-viscosity}
    Take the same assumptions on \(U_1,U_2\) as in Lemma \ref{lem:pos-wealth}. The value function \(V\) in \eqref{eq:gen-value} is the unique viscosity solution to \eqref{eq:gen-hjb} over the any  \(\bar{\mathcal D}=[0,T]\times \mathcal D_1\times \mathcal D_2\subset\calD\) with \(\mathcal D_i\) compact, \(i=1,2\).
\end{proposition}
\begin{proof}
    First, notice that under the assumptions on admissible control specified in \eqref{eq:adm-set} as well as those on the utility functions specified in \eqref{eq:gen-util}, the continuity of the value function \(V\) in \eqref{eq:gen-value} over the domain \(\calD\) can be established following the classical results of \cite{lions1983hjbI,lions1982hjb}. Therefore, \(V\) is bounded and uniformly continuous on \(\bar{\calD}\). Combining Proposition \ref{prop:gen-dpp} and similar arguments of It\^{o}'s formula in its proof, the viscosity solution property in Definition \ref{def:visc-soln} can be established. The uniqueness result follows a classical comparison principal \cite[Theorem V9.1]{fleming2006controlled}.
\end{proof}

Given \(V\) being a \(\mathcal{C}^{1,2,1}\left([0,T)\times(0,\infty)\times[0,1]\right)\bigcap\calC\left([0,T]\times(0,\infty)\times[0,1]\right)\), define the Hamiltonian as
\[H(t,x,z,\alpha,c,p,q)\coloneqq zU_1(c)+\left\{x\left[\alpha\mu+(1-\alpha)r\right]-c\right\}p+\frac{\sigma^2\alpha^2}{2}x^2q.\]
Then, we have that for any \((t,x,z)\in[0,T)\times(0,\infty)\in[0,1]\),
\[
    (\alpha_t^*,c_t^*)=(\alpha^*(t,x,z),c^*(t,x,z))=\argmax_{(\alpha,c)\in\mathcal K}H\Big(t,x,z,\alpha,c,\partial_xV(t,x,z),\partial_x^2V(t,x,z)\Big),
\] 
where
\begin{align}
    \alpha^*_t&=-\frac{\partial_xV(t,x,z)}{\partial_x^2V(t,x,z)}\cdot\frac{(\mu-r)}{\sigma^2x}\wedge M,\label{eq:gen-opt-alpha}\\
    c^*_t&=\argmax_{c\in[0,M]}\Big\{-\partial_xV(t,x,z)c+zU_1(c)\Big\};\label{eq:gen-opt-c}
\end{align}
note that by Lemma \ref{lem:gen-v-1}, \(\alpha_t^*>0\).


\subsubsection{The Inverse Problem: Identifiability of the Utility Functions}
In this section, we focus on the ``inverse'' problem with respect to the optimal asset allocation-consumption scenario and study the ``identifiability'' of the utility functions as well as the discounting scheme out of the optimal investment policies. More specifically, we assume the client provides her decision policies (in the sense of the allocation-consumption processes) to the inference agent.

Following practical protocols, we assume that the inference agent does not know the discounting scheme \(\dot{\beta}\) nor the utility functions \(U_1\) and \(U_2\). Nevertheless, the inference agent tries to infer these characteristic functions based on available information, namely the joint allocation-consumption process (i.e., control policy) provided by the client. 

To start, let 
\begin{equation}
    \label{eq:gen-inv-opt-ctrl}
    ({\bar\alpha},{\bar c}):[0,T]\times(0,\infty)\times[0,1]\to\mathcal{K}
\end{equation}
be some allocation and consumption policies of a client such that 
\[\left(\bar{\balpha}=\left\{\bar\alpha_t\right\}_{t\in[0,T]}=\left\{\bar\alpha(t,\ctrldX[\bar{\balpha},\bar{\bc}]_t,\beta_t)\right\}_{t\in[0,T]},\bar{\bc}=\left\{\bar c_t\right\}_{t\in[0,T]}=\left\{\bar c(t,\ctrldX[\bar{\balpha},\bar{\bc}]_t,\beta_t)\right\}_{t\in[0,T]}\right)\in\calA,\]
where \(\left(\left\{\beta_t\right\}_{t\in[0,T]},\left\{\ctrldX[\bar{\balpha},\bar{\bc}]_t\right\}_{t\in[0,T]}\right)\) solves
\begin{align*}
    &d\beta_t=\dot{\beta}_tdt,\\
    &d\ctrldX[\bar{\balpha},\bar{\bc}]_t=\left\{\ctrldX[\bar{\balpha},\bar{\bc}]_t\left[\bar\alpha\left(t,\ctrldX[\bar{\balpha},\bar{\bc}]_t,\beta_t\right)(\mu-r)+r\right]-\bar c\left(t,\ctrldX[\bar{\balpha},\bar{\bc}]_t,\beta_t\right)\right\}dt+\sigma \bar\alpha\left(t,\ctrldX[\bar{\balpha},\bar{\bc}]_t,\beta_t\right)\ctrldX[\bar{\balpha},\bar{\bc}]_tdW_t,
\end{align*}
for any given \((\beta_0,\ctrldX[\bar{\balpha},\bar{\bc}]_0)\in[0,1]\times(0,\infty)\), and 
\[\left(\bar{\balpha},\bar{\bc}\right)\in\argmax_{(\balpha,\bc)\in\calA}J(t,x,z,\balpha,\bc),\ \ \forall(t,x,z)\in[0,T]\times(0,\infty)\times[0,1]\]
subject to \eqref{eq:gen-wealth} and \eqref{eq:gen-discnt} . Write
\begin{equation}
    \label{eq:inv-value}
    \bar{V}(t,x,z)=J(t,x,z,\bar\balpha,\bar\bc),\ \ \forall (t,x,z)\in[0,T]\times(0,\infty)\times[0,1].
\end{equation}
Note that the inference agent has full access to \eqref{eq:gen-inv-opt-ctrl}.

\begin{theorem}[Identifiability]\label{thm:gen-idtf}
Assume that 
\begin{enumerate}
    \item \(\bar\alpha,\bar c\in\calC^{1,1,1}([0,T)\times(0,\infty)\times[0,1])\bigcap\calC^{0}([0,T]\times(0,\infty)\times[0,1])\);
    \item \(\bar\alpha(t,x,z)\in(0,M)\) for all \((t,x,z)\in[0,T]\times(0,\infty)\times[0,1]\);\label{assp:2}
    \item for any \((t,z)\in[0,T]\times[0,1]\),
    \[\bar c(t,x,z)<x,\ \  \forall x>0;\]
    \item both \(\bar\alpha(T,\cdot,\cdot)\) and \(\bar c(T,\cdot,\cdot)\) are ``\(z\)-free'', denoted by 
    \[\bar\alpha(T,x,z)\equiv\bar\alpha_T(x),\ \ \bar c(T,x,z)\equiv\bar c_T(x),\ \ \forall (x,z)\in(0,\infty)\times[0,1];\]
    \item for \(x\in(0,\infty)\), \(\bar\alpha_T(x)>0\), \(\bar c_T\) is invertible, and the following difference for any \((t,z)\in[0,T)\times(0,1]\),
    \[\Delta(t,x,z)\coloneqq\int_x^1\frac{dy}{y\bar\alpha(t,y,z)}-\int_{\bar c_T^{-1}(\bar c(t,x,z))}^1\frac{dy}{y\bar\alpha_T(y)}\]
    depends only on \((t,z)\), namely \(\Delta(t,x,z)\equiv\Delta(t,z)\). \label{assp:5}
\end{enumerate}
Then both the discounting scheme characterized by \(\dot{\beta}\) and the utility functions \(U_i\in\mathcal U\) for \(i=1,2\), with \(U_1(0)=0\) and \(U_2(0)=-\infty\), are identifiable up to an affine transform. 
\end{theorem}

\begin{remark}
     This result is also consistent with the finding in \cite{cao2021identifiability} that the identifiability of the unknown utility function in an inverse optimal control problem is equivalent to the identifiability of the corresponding value function under the observed optimal policy. Assumptions \ref{assp:2} and \ref{assp:5} enunciate the precise dependency of value function \(\bar V\) in \eqref{eq:inv-value} and the observed policy \((\bar\alpha,\bar c)\).
\end{remark}

\begin{proof}
    First, by Assumption \ref{assp:2}, for all  \((t,x,z)\in[0,T]\times(0,\infty)\times[0,1]\), \eqref{eq:gen-hjb} is equivalent to
    \begin{equation}
    \label{eq:gen-hjb-pde}
    \left\{\begin{aligned}
        &\partial_tV(t,x,z)+\dot{\beta}_t\partial_zV(t,x,z)+rx\partial_xV(t,x,z)\\
        &\hspace{100pt}-\frac{(r-\mu)^2}{2\sigma^2}\frac{\left[\partial_xV(t,x,z)\right]^2}{\partial_x^2V(t,x,z)}=zU^*_1\left(\frac{\partial_xV(t,x,z)}{z}\right),\ \ t\in[0,T);\\
        &V(T,x,z)=zU_2(x),
    \end{aligned}
    \right.
    \end{equation}
    where \(U_1^*\) is the Legendre transform of the concave utility function \(U_1:[0,\infty)\to\R\),
    \[U_1^*(\kappa)\coloneqq \inf_{c\in[0,M]}\left\{\kappa c-U_1(c)\right\},\ \ \forall\kappa\in\R.\]
    Now, we construct the value function \(\bar V\) in \eqref{eq:inv-value} from \eqref{eq:gen-hjb-pde}. By \eqref{eq:gen-opt-alpha} and Assumption \ref{assp:2},
    \[
    \bar\alpha_T(x)=-\frac{U_2'(x)}{U_2''(x)}\frac{\mu-r}{\sigma^2x}\implies U_2(x)=k_1\int_1^x\exp\left\{\int_y^1\frac{\mu-r}{\sigma^2u\bar\alpha_T(u)}du\right\}dy+k_2,
    \]
    for some \(k_1,k_2>0\); in particular, 
    \[U_2'(x)=k_1\exp\left\{\int_x^1\frac{\mu-r}{y\bar\alpha_T(y)}dy\right\}.\]
    By \eqref{eq:gen-opt-c}, for any \(x>0\),
    \[U_1'(\bar c_T(x))=U_2'(x)\implies U_1'(x)=U_2'(\bar c_T^{-1}(x))=k_1\exp\left\{\int_{\bar c_T^{-1}(x)}^1\frac{\mu-r}{\sigma^2y\bar\alpha_T(y)}dy\right\} \]
    and 
    \[U_1(x)=k_1\int_1^x\exp\left\{\int_{\bar c_T^{-1}(y)}^1\frac{\mu-r}{\sigma^2u\bar\alpha_T(u)}du\right\}dy+k_3\]
    for some \(k_3>0\).
    
    For any \((t,z)\in[0,T)\times(0,1]\) and \(x>0\), by \eqref{eq:gen-opt-alpha}--\eqref{eq:gen-opt-c},
    \[\begin{cases}
    \bar\alpha(t,x,z)=-\frac{\partial_x\bar V(t,x,z)}{\partial_x[\partial_x\bar V(t,x,z)]}\frac{\mu-r}{\sigma^2x},\\
    \partial_x\bar V(t,x,z)=zU_1'(\bar c(t,x,z));
    \end{cases}\implies\partial_x \bar V(t,x,z)=K_1(t,z)\exp\left\{\int_x^1\frac{\mu-r}{\sigma^2y\bar\alpha(t,y,z)}dy\right\},\]
    where 
    \[K_1(t,z)=k_1z\exp\{-\frac{\mu-r}{\sigma^2}\Delta(t,z)\}.\]
    Rewrite \eqref{eq:gen-hjb-pde} as
    \[
    \partial_t\bar V(t,x,z)+\dot{\beta}_t\partial_z\bar V(t,x,z)=-\left\{\left[r+\frac{\bar\alpha(t,x,z)(\mu-r)}{2}\right]x-\bar c(t,x,z)\right\}\partial_x\bar V(t,x,z)-zU_1(\bar c(t,x,z)).
    \]
    Differentiating with respect to \(x\) on both sides, we have 
    \[\dot{\beta}_t=-\frac{\partial_t[\partial_x\bar V(t,x,z)]+\partial_x\left\{\left\{\left[r+\frac{\bar\alpha(t,x,z)(\mu-r)}{2}\right]x-\bar c(t,x,z)\right\}\partial_x\bar V(t,x,z)+zU_1(\bar c(t,x,z))\right\}}{\partial_z[\partial_x\bar V(t,x,z)]}.\]
\end{proof}

We conclude the analysis on finite-time horizon by discussing a special case with an explicit solution. 
\begin{example}
Set $\beta_0=1$, $\dot\beta_t = 0$ $(0\leq t \leq T)$, {$U_1(c) = 0$ and a CRRA (power) utility $U_2(x) = \frac{x^\theta}{\theta}$ with $0<\theta<1$}.  Also set the constant $M$ such that $M > \frac{\mu-r}{\sigma^2}$.
In this case, {we face a classic control problem that is time-consistent. Hence state augmentation is not necessary.}  In addition, both the optimal control and the inverse problem have explicit representations. The goal here is to identify the parameter $\theta$ from the client.

Consequently, define the value function:
    \begin{eqnarray}
        V(t,x) = \sup_{\alpha \in \mathcal{A}}\mathbb{E}\Big[ U_2(X_T)\,\left|\,X_t = x\Big].\right.
    \end{eqnarray}
  The value function satisfies the following HJB equation:
    \begin{eqnarray}
        -\partial_t V - \sup_{\alpha \in \mathcal{A}} \Big[\mathcal{L}^\alpha V(t,x)\Big] = 0,
    \end{eqnarray}
    with boundary condition $V(T,x) = U_2(x)= \frac{x^\theta}{\theta}$,
    where the generator is defined as $\mathcal{L}^\alpha V(t,x) = x\Big( \alpha \mu + (1-\alpha)r\Big)\partial_{x}V+ \frac{1}{2} x^2 \alpha^2\sigma^2 \partial_{x}^2 V$.
   The optimal policy follows:
    \begin{eqnarray}
        \bar\alpha(t,x) = -M\vee \left(- \frac{(\mu-r)\partial_x V}{x \sigma^2 \partial_x^2 V}\right)\wedge M.
    \end{eqnarray}
  For now assume that $-M\leq \left(- \frac{(\mu-r)\partial_x V}{x \sigma^2 \partial_x^2 V}\right)\leq M$ (to be checked later). Then plugging it back into the HJB equation, we have
    \begin{eqnarray}
        -\frac{\partial V}{\partial t} = x r \partial_{x}V - \frac{1}{2} \frac{(\mu-r)(\partial_{x}V)^2}{\sigma^2 \partial^2_{x}V}
    \end{eqnarray}
    with boundary condition $V(T,x) =  \frac{x^\theta}{\theta}$.
 We take the ansatz $V(t,x) = \phi(t)\frac{x^\theta}{\theta}$. Hence $\phi(t)$ satisfies:
\begin{eqnarray}
    \phi'(t) = \rho \phi(t) = 0, \quad \phi(T)=1,
\end{eqnarray}
where $\rho = p\times \sup_{\alpha\in[-M,M]}\Big[ \alpha(\mu-r) + r -\frac{1}{2}a^2(1-p)\sigma^2\Big]$. Hence $\bar\alpha = \frac{
\mu-r}{\sigma^2(1-\theta)}\in[-M,M]$. In this case, it is obvious that condition $-M\leq \left(- \frac{(\mu-r)\partial_x V}{x \sigma^2 \partial_x^2 V}\right)\leq M$ is satisfied. Therefore we can recover the preference parameter by using $1-\frac{\mu-r}{\bar \alpha \sigma^2}$.
 \end{example}

\subsection{Infinite-time Horizon}\label{subsubsec:gen-discnt-inf}
{Now we shift our focus to an infinite-time horizon setting that accommodates a long-run investment planning scenario. 

Recall that the investing client is holding a general discounting scheme \(\beta=\{\beta_t\}_{t\geq0}\) where}
\begin{itemize}
    \item \(\beta_t\in[0,1]\) for all \(t\in[0,\infty)\) such that \(\lim_{t\to\infty}\beta_t=0\); and
    \item there exists \(\dot{\beta}:[0,\infty)\to\R\) such that \(\dot{\beta}\) is integrable on \([0,t]\) with \(\beta_t=\int_0^t\dot{\beta}_sds+\beta_0\) for any \(t>0\).
\end{itemize}

For any \((t,x,z)\in[0,\infty)\times\R\times[0,1]\), define the total reward function as 
\begin{equation}
    \label{eq:acc-util-inf}
    J_\infty(t,x,z,\balpha,\bc)\coloneqq\E\left[\int_t^\infty\beta_sU_1(c_s)ds\biggl|X_t=x,\beta_t=z\right]
\end{equation}
subject to  \eqref{eq:gen-wealth} and \eqref{eq:gen-discnt}, under a given allocation process \(\balpha=\{\alpha_t\}_{t\geq0}\) and a given consumption process \(\bc=\{c_t\}_{t\geq0}\) with \(c_t\geq0\). For any \((t,x,z)\in[0,\infty)\times\R\times[0,1]\), define the value function as follows,
\begin{equation}
    \label{eq:gen-value-inf}
    V_\infty(t,x,z)=\sup_{(\balpha,\bc)\in\calA}J_\infty(t,x,z,\balpha,\bc),\,t\in[0,t);\ \ \lim_{t\to\infty}V_\infty(t,x,z)=0,
\end{equation}
subject to \eqref{eq:gen-wealth} and \eqref{eq:gen-discnt}. 


It is easy to show that the value function \(V_\infty\) in \eqref{eq:gen-value-inf} will have similar results as specified in Section \ref{subsec:gen-finite} therefore here we state these results without proofs.
First, we have the necessary condition for \(V_\infty\).
\begin{proposition}\label{prop:gen-dpp-inf}
For any \((t,x,z)\in[0,\infty)\times\R\times[0,1]\) and \(\tau\in\bar{\mathbb T}_t\) where \(\bar{\mathbb T}_t\) denotes all \(\{\bar\calF_t\}_{t\geq0}\)-adapted stopping times \(\tau\) such that \(\tau\in[t,\infty)\) a.s.. Then the value function \(V_\infty\) defined in \eqref{eq:gen-value-inf} satisfies
\begin{equation}
    \label{eq:gen-dpp-inf}
    \tag{DPP'}
    V_\infty(t,x,z)=\sup_{(\balpha,\bc)\in\calA}\E\left[\int_t^{\tau}\beta_sU_1(c_s)ds+V_\infty(\tau, \ctrldX_{\tau},\beta_{\tau})\,\biggl|\,\ctrldX_t=x,\beta_t=z\right].
\end{equation}
\end{proposition}
The corresponding HJB equation is given by
\begin{equation}
    \label{eq:gen-hjb-inf}
    \tag{HJB'}\left\{
    \begin{aligned}
    &\partial_tV_\infty(t,x,z)+\dot{\beta}_t\partial_zV_\infty(t,x,z)+rx\partial_xV_\infty(t,x,z)+\sup_{\alpha\in\R,c\geq0}\Big\{zU_1(c)+\infgen V_\infty(t,x,z)\Big\}=0,\ \ t\in[0,\infty);\\
    &\lim_{t\to\infty}V_\infty(t,x,z)=0.
    \end{aligned}\right.
\end{equation}

Likewise, combining Proposition \ref{prop:gen-dpp-inf} and It\^{o}'s formula, we have the following result.
\begin{proposition}
    \label{prop:gen-viscosity-inf}
    If the value function \(V_\infty\) in \eqref{eq:gen-value-inf} is jointly continuous on \(\calD'=[0,\infty)\times(0,\infty)\times[0,1]\), then it is a viscosity solution to \eqref{eq:gen-hjb-inf} over the domain \(\calD'\).
\end{proposition}

With a classical solution to \eqref{eq:gen-hjb-inf}, we have the following verification theorem serving as sufficient conditions for \(V_\infty\).
\begin{proposition}
    \label{prop:gen-vrf-inf}
    Suppose that \(U_1:[0,\infty)\to\R^+\in\mathcal U\) continuous at \(0\). Let \(w:[0,\infty)\times\R\times[0,1]\to\R\) be a function such that 
    \[w\in\calC^{1,2,1}\left([0,\infty)\times\R\times[0,1]\right),\]
    and there exists a constant \(C>0\) with
    \[w(t,x,z)\leq C(1+|x|^2),\ \ \forall (t,x,z)\in[0,\infty)\times\R\times[0,1].\]
    \begin{enumerate}
        \item Assume that for any \((\alpha,c)\in\mathcal K\), 
        \[
        \left\{
        \begin{aligned}
            &\partial_tw(t,x,z)+\dot{\beta}_t\partial_zw(t,x,z)+rx\partial_xw(t,x,z)+zU_1(c)+\infgen w(t,x,z)\leq0,\\
            &\hspace{180pt}\forall (t,x,z)\in[0,\infty)\times\R\times[0,1];\\
            &\lim_{t\to\infty}w(t,x,z)=\infty, \ \  \forall (x,z)\in\R\times[0,1].
        \end{aligned}\right.
        \]
        Then \(w\geq V_\infty\) on \([0,\infty)\times\R\times[0,1]\).
        \item Assume further that there exists \(\hat{\alpha}:[0,\infty)\times\R\times[0,1]\to[-M,M]\) and \(\hat{c}:[0,\infty)\times\R\times[0,1]\to[0,M]\) such that 
        \[
        \left\{
        \begin{aligned}
            &\partial_tw(t,x,z)+\dot{\beta}_t\partial_zw(t,x,z)+rx\partial_xw(t,x,z)+zU_1(\hat c(t,x,z))+\infgen[\hat{\alpha}(t,x,z),\hat{c}(t,x,z)]w(t,x,z)=0,\\
            &\hspace{180pt}\forall (t,x,z)\in[0,T)\times\R\times[0,1];\\
            &\lim_{t\to\infty}w(t,x,z)=\infty, \ \  \forall (x,z)\in\R\times[0,1],
        \end{aligned}\right.
        \]
        also, with \(\beta_t=\beta_0+\int_0^t\dot{\beta}_sds\in[0,1]\) for all \(t\geq0\), the following SDE,
        \[dX_t=
        \left\{X_t\left[\hat{\alpha}(t,X_t,\beta_t)(\mu-r)+r\right]-\hat{c}(t,X_t,\beta_t)\right\}dt
        +\sigma\hat{\alpha}(t,X_t,\beta_t)X_tdW_t, 
        \]
        admits a unique solution \(\ctrldX[\hat{\balpha},\hat{\bc}]\) given \(X_0=x\) for any \(x\in\R\), and 
        \[\left(\hat{\balpha}=\left\{\hat{\alpha}_t\right\}_{t\geq0}=\left\{\hat{\alpha}(t,\ctrldX[\hat{\balpha},\hat{\bc}]_t,\beta_t)\right\}_{t\geq0}, \hat{\bc}=\left\{\hat{c}_t\right\}_{t\geq0}=\left\{\hat{c}(t,\ctrldX[\hat{\balpha},\hat{\bc}]_t,\beta_t)\right\}_{t\geq0}\right)\in\calA.\]
        Then \(w= V_\infty\) on \([0,\infty)\times\R\times[0,1]\), with \((\hat{\balpha},\hat{\bc})\) being an optimal joint allocation-consumption process.
    \end{enumerate}
\end{proposition}
Given that \(V_\infty\in\calC^{1,2,1}(\calD')\) being a classical solution to \eqref{eq:gen-hjb-inf}, then the optimal policy is given by 
\begin{align}
    \bar\alpha^*(t,x,z)&=-M\vee-\frac{\partial_xV_\infty(t,x,z)}{\partial_x^2V_\infty(t,x,z)}\cdot\frac{(\mu-r)}{\sigma^2x}\wedge M,\label{eq:gen-opt-alpha-inf}\\
     \bar c^*(t,x,z)&=\argmax_{c\in[0,M]}\Big\{-\partial_xV_\infty(t,x,z)c+zU_1(c)\Big\}.\label{eq:gen-opt-c-inf}
\end{align}

Accordingly, for the inverse problem, we also have the following identifiability result.
\begin{theorem}\label{thm:gen-idtf-inf}
Assume that 
\begin{enumerate}
    \item \(\bar\alpha,\bar c\in\calC^{1,1,1}([0,\infty)\times(0,\infty)\times[0,1])\bigcap\calC^{0}([0,\infty)\times(0,\infty)\times[0,1])\);
    \item \(\bar\alpha(t,x,z)\in(-M,M)\) for all \((t,x,z)\in[0,T]\times(0,\infty)\times[0,1]\);
    \item for any \((t,z)\in[0,\infty)\times[0,1]\),
    \[\bar c(t,x,z)<x,\ \  \forall x>0;\]
    \item \(\exists(t_0,z_0)\in[0,\infty)\times(0,1]\) such that  \(\bar c_0(\cdot)\coloneqq \bar c(t_0,\cdot,z_0)\) is invertible, and the following difference for any \((t,x,z)\in[0,T)\times(0,\infty)\times(0,1]\),
   \[\Delta(t,x,z)\coloneqq\int_x^1\frac{dy}{y\bar\alpha(t,y,z)}-\int_{\bar c_0^{-1}(\bar c(t,x,z))}^1\frac{dy}{y\bar\alpha_T(y)}\]
   depends only on \((t,z)\), namely, \(\Delta(t,x,z)\equiv\Delta(t,z)\).
\end{enumerate}
Then both the discounting scheme characterized by \(\dot{\beta}\) and the utility function \(U_1\in\mathcal U\) with \(U_1(0)=0\) are identifiable up to an affine transform.
\end{theorem}

\section{Discrete-time MDP with Entropy Regularization }\label{sec:discrete-time}
The continuous-time framework in Section \ref{sec:continuous_time} emphasizes the well-definedness of the mathematical framework when the client is subject to a generic discounting scheme with a time-varying rate, and it outlines conditions necessary for ensuring identifiability for both the utility functions and the discounting scheme. Building on these insights, this section explores a practical scenario focusing on the inference procedure. We adopt a parametric framework in which the client utilizes an exponential discounting scheme, parameterized by $\bar{\rho}$, alongside a utility function parameterized by $\bar{\theta} \in \mathbb{R}^d$. The client's preference parameter is summarized as $(\Bar{\rho},\Bar{\theta}
)$, which is unknown to the inference agent. The inference agent employs a maximum likelihood estimation method to infer the parameters $(\bar{\rho}, \bar{\theta})$. This analysis is conducted within a discrete-time MDP setting under the regularization of a Shannon entropy type. It encourages the client to fully explore the state-action space and introduces smoothness to the analysis at the same time; see \cite{haarnoja2017reinforcement}, for instance.

Mathematically, let us consider the entropy regularized MDP with state space $\calS$ and action space $\calA$, which can be finite or infinite. The state process follows $s_{t+1}\sim P(\cdot|s_t,a_t)$, with $P:\mathcal{S}\times \mathcal{A}\rightarrow\mathcal{P}(\mathcal{A})$ the transition kernel that maps from the joint state-action space to the distribution over the state space. After taking action $a$ at state $s$, we assume the client receives a deterministic reward $R(s,a)\in[0,1]$. Throughout the remainder of this paper, we will use the notation $\SumInt$ to denote the summation or integration over the action space, emphasizing that our framework accommodates both finite and infinite action spaces.

Under a generic preference parameter $(\theta,\rho)$, consider the entropy regularized objective:
\begin{align}
Q_{\rho,\theta}^*(s,a)=\max_\pi\mathbb{E}^\pi\left.\left[\sum_{t=1}^\infty \rho^ t \Big(U_\theta\big(R(s_t,a_t)\big)+\mathcal{H}\big(\pi(\cdot|s_t)\big)\Big)\right|s_0=s,a_0=a,a_t\sim \pi(s_t)\right],
\end{align}
where $\calH\big(\pi(\cdot|s)\big):= -\SumInt_{a\in\calA}\pi(a|s)\log\big(\pi(a|s)\big)$ is the Shannon's entropy. The optimal policy is then given by:
\begin{align}
\label{equ:optimal_policy}
\pi_{\rho,\theta}(a|s) =\frac{e^{Q_{\rho,\theta}^*(s,a)}}{\SumInt_{a\in\calA} e^{Q_{\rho,\theta}^*(s,a')}},
\end{align}
and the soft Bellman equation holds: 
\begin{align}
\label{equ:soft_Bellman_equation}
V_{\rho,\theta}^*(s) = \log\Big(\SumInt_{a\in\calA} e^{Q_{\rho,\theta}^*(s,a)} \Big).
\end{align}

\subsection{Maximum Likelihood Estimation}\label{sec:mle}
With a trajectory $\tau=\{(s_t,a_t)\}_{t=0}^\infty$ following the client's policy $\pi_{\Bar{\rho},\Bar{\theta}}$, we adopt a {\it maximum likelihood estimation method} to infer the client's preference parameter $(\Bar{\rho},\Bar{\theta})$, which is unknown to the inference agent. Specifically, the discounted likelihood of a trajectory $\tau=\{(s_t,a_t)\}_{t=0}^\infty$ following the client's policy $\pi_{\Bar{\rho},\Bar{\theta}}$ is defined as
\begin{eqnarray}
&&\E_{\tau\sim\pi_{\Bar{\rho},\Bar{\theta}}}\left[\log \left(\prod_{t=0}^\infty (P(s_{t+1}|s_t,a_t) \pi_{\rho,\theta}(a_t|s_t))^{\gamma^t}\right)\right]\nonumber\\
&&\qquad=\E_{\tau\sim\pi_{\Bar{\rho},\Bar{\theta}}}\Big[\sum_{t=0}^\infty \gamma^t\log \pi_{\rho,\theta}(a_t|s_t) \Big] + \E_{\tau\sim\pi_{\Bar{\rho},\Bar{\theta}}}\Big[\sum_{t=0}^\infty \gamma^t\log  P(s_{t+1}|s_t,a_t)  \Big],\label{eq:likelihood_obj}
\end{eqnarray}
where the notation $\tau\sim\pi_{\Bar{\rho},\Bar{\theta}}$ represents that the trajectory $\tau$ is sampled from applying policy $\pi_{\Bar{\rho},\Bar{\theta}}$ and $\gamma$ is a discount factor specified by the inference agent, which is potentially {\it different} from $\Bar\rho$.

\begin{remark} Note that in our case $\gamma\neq \Bar{\rho}$ because $\Bar{\rho}$ is the client's discount factor and is unknown to the inference agent.
This distinguishes us from the usual IRL literature, where $\Bar{\rho} $ is always assumed to be known \citep{bloem2014infinite}. For example, \citet{zeng2022maximum} studied the IRL problem using a maximum likelihood estimator by setting  $\gamma = \rho$ in \eqref{eq:likelihood_obj} and showed that their algorithm
converges to a stationary point with a finite-time guarantee. Note that this stationary point may not be the ground-truth solution.
\end{remark}

The maximum likelihood inference problem can be written as:
\begin{align}
\max_{(\rho,\theta)\in\Theta}\quad & \mathcal{L}(\rho,\theta):= 
\E_{\tau\sim\pi_{\Bar{\rho},\Bar{\theta}}}\Big[\sum_{t=0}^\infty \gamma^t \log \pi_{\rho,\theta}(a_t|s_t)\Big],
\end{align}
where $\pi_{\rho,\theta}$ is the optimal policy under the preference parameter $(\rho,\theta)$ defined in \eqref{equ:optimal_policy}. Here for simplicity we set $\Theta := (0,1)\times \mathbb{R}^{d}$. The maximum likelihood problem is to find a preference parameter $(\Bar{\rho},\Bar{\theta})$ that generates the client's trajectory with the highest likelihood. 

The goal is to investigate the landscape of the log-likelihood function $\mathcal{L}(\rho,\theta)$ with respect to $(\rho,\theta)$ and understand the possibility of recovering $(\Bar{\rho},\Bar{\theta})$, which is also referred to as the inverse problem. To proceed, we first show that $(\Bar{\rho},\Bar{\theta})$ is a stationary point of the likelihood function (see Proposition \ref{lemma:gradient_L}) and then show that likelihood function is concave near $(\Bar{\rho},\Bar{\theta})$ (see Theorem \ref{lemma:landscape}). Interestingly, the results of the landscape analysis is {\it independent of} the choice of $\gamma$, making our proposed method robust and practical.

\begin{proposition}
\label{lemma:gradient_L}
   It holds that
    \begin{eqnarray}\label{eq:first-order}
        \nabla_\theta \mathcal{L}(\rho,\theta)
        \Big|_{(\rho,\theta) = (\Bar{\rho},\Bar{\theta})} = 0,\quad \nabla_\rho \mathcal{L}(\rho,\theta)\Big|_{(\rho,\theta) = (\Bar{\rho},\Bar{\theta})} = 0. 
    \end{eqnarray}
\end{proposition}
Proposition \ref{lemma:gradient_L} suggests that the gradient of the likelihood function equals zero at the client's preference parameter value $(\Bar{\rho},\Bar{\theta})$, and hence $(\Bar{\rho},\Bar{\theta})$ is a stationary point of the likelihood function $\mathcal{L}(\rho,\theta)$.
\begin{proof}
Our proof can be divided into three steps. We first provide some useful formulas regarding the first-order and second-order derivatives of $Q,V$ with respect to $(\rho,\theta)$. With such formulas, we next derive the derivatives of the log-likelihood function. Finally, we show that \eqref{eq:first-order} holds.
\vspace{10pt}

\noindent\underline{Step 1.} To begin with, for any {$(\rho,\theta)\in \Theta$}, for any $(s_t,a_t)\in\calS\times\calA$,
\begin{align}
\nabla_\theta Q_{\rho,\theta}(s_t,a_t) 
&= \nabla_\theta U_\theta(R(s_t,a_t))+\rho \E_{s_{t+1}\sim P(\cdot|s_t,a_t)}\Big[\nabla_\theta V_{\rho,\theta}(s_{t+1}) \Big]\nonumber\\
&= \nabla_\theta U_\theta(R(s_t,a_t))+\rho \E_{s_{t+1}\sim P(\cdot|s_t,a_t)}\Big[\nabla_\theta \log({\SumInt_a e^{Q_{\rho,\theta}(s_{t+1},a)}}  ) \Big]\label{equ:likelihood_gradient_0.1}\\
&= \nabla_\theta U_\theta(R(s_t,a_t))+\rho \E_{s_{t+1}\sim P(\cdot|s_t,a_t)}\Big[
\SumInt_a \pi_{\rho,\theta}(a|s_{t+1})  \nabla_\theta  Q_{\rho,\theta}(s_{t+1},a) \Big]\label{equ:likelihood_gradient_0.2}\\
&= \nabla_\theta U_\theta(R(s_t,a_t))+\rho \E_{s_{t+1}\sim P(\cdot|s_t,a_t),a_{t+1}\sim \pi_{\rho,\theta}(\cdot|s_{t+1})}\Big[ \nabla_\theta  Q_{\rho,\theta}(s_{t+1},a_{t+1}) \Big],\label{equ:likelihood_gradient_1}
\end{align}
where \eqref{equ:likelihood_gradient_0.1} holds by the soft Bellman equation and \eqref{equ:likelihood_gradient_0.2} holds  because $\pi_{\rho,\theta}$ is the optimal policy. 
Applying \eqref{equ:likelihood_gradient_1} recursively yields:
\begin{eqnarray}\label{eq:recursion_result}
\nabla_\theta Q_{\rho,\theta}(s_t,a_t) = \E_{\tau'\sim \pi_{\rho,\theta}}\Big[\sum_{k=t}^\infty \rho^{k-t}\nabla_\theta U_\theta(R(s_{k}',a_{k}'))\Big|s_t,a_t  \Big],
\end{eqnarray}    
where $\tau'=\{s'_k,a'_k\}_{k=0}^{\infty}$ denotes a trajectory following $\pi_{\rho,\theta}$. Similarly,
\begin{align}
\nabla_\rho Q_{\rho,\theta}(s_t,a_t) 
&= \nabla_\rho U_\theta(R(s_t,a_t))+\nabla_\rho (\rho \E_{s_{t+1}\sim P(\cdot|s_t,a_t)}\Big[ V_{\rho,\theta}(s_{t+1}) \Big])\label{equ:likelihood_gradient_0.25}\\
&=  \E_{s_{t+1}\sim P(\cdot|s_t,a_t)}\Big[ V_{\rho,\theta}(s_{t+1}) \Big] + \rho\E_{s_{t+1}\sim P(\cdot|s_t,a_t)}\Big[ \nabla_\rho \log({\SumInt_a e^{Q_{\rho,\theta}(s_{t+1},a)}}  )  \Big]\label{equ:likelihood_gradient_0.3}\\
&=  \E_{s_{t+1}\sim P(\cdot|s_t,a_t)}\Big[ V_{\rho,\theta}(s_{t+1}) \Big]+ \rho\E_{s_{t+1}\sim P(\cdot|s_t,a_t)}\Big[
\SumInt_a \pi_{\rho,\theta}(a|s_{t+1})  \nabla_\rho  Q_{\rho,\theta}(s_{t+1},a) \Big]\label{equ:likelihood_gradient_0.4}\\
&=  \E_{s_{t+1}\sim P(\cdot|s_t,a_t)}\Big[ V_{\rho,\theta}(s_{t+1}) \Big]+ \rho\E_{s_{t+1}\sim P(\cdot|s_t,a_t),\ a_{t+1}\sim \pi_{\rho,\theta}(\cdot|s_{t+1})}\Big[ \nabla_\rho  Q_{\rho,\theta}(s_{t+1},a_{t+1}) \Big],\label{equ:likelihood_gradient_0}
\end{align}
where \eqref{equ:likelihood_gradient_0.3} holds by the Bellman equation and \eqref{equ:likelihood_gradient_0.4} holds because $\pi_{\rho,\theta}$ is the optimal policy.  Applying \eqref{equ:likelihood_gradient_0} recursively yields:
\begin{eqnarray}\label{eq:recursion_result2}
\nabla_\rho Q_{\rho,\theta}(s_t,a_t) = \E_{\tau'\sim \pi_{\rho,\theta}}\Big[\sum_{k=t+1}^\infty \rho^{k-t-1}V_{\rho,\theta}(s'_{k})\Big|s_t,a_t  \Big].
\end{eqnarray}   
Furthermore, we have for any $s_t\in\calS$,
\begin{align}
&\nabla_\theta V_{\rho,\theta}(s_t) = \nabla_\theta (\log \SumInt_a e^{Q_{\rho,\theta}(s_t,a)})\nonumber\\
&= \E_{a_t\sim\pi_{\rho,\theta}(\cdot|s_t)}\Big[\nabla_\theta {Q_{\rho,\theta}(s_t,a_t)}\Big]= \E_{\tau'\sim \pi_{\rho,\theta}}\Big[\sum_{k=t}^\infty \rho^{k-t}\nabla_\theta U_\theta(R(s_{k}',a_{k}'))\Big|s_t \Big],\label{equ:likelihood_gradient_3.5}
\end{align}
where the last equation holds by \eqref{eq:recursion_result}. In addition,
\begin{align}
&\nabla_\rho V_{\rho,\theta}(s_t)
= \nabla_\rho (\log \SumInt_a e^{Q_{\rho,\theta}(s_t,a)})\nonumber\\
&= \E_{a_t\sim\pi_{\rho,\theta}(\cdot|s_t)}\Big[\nabla_\rho {Q_{\rho,\theta}(s_t,a_t)}\Big]= \E_{\tau'\sim \pi_{\rho,\theta}}\Big[\sum_{k=t+1}^\infty \rho^{k-t-1}V_{\rho,\theta}(s_{k}')\Big|s_t \Big],\label{equ:likelihood_gradient_3.6}
\end{align}
where the last equation holds by \eqref{eq:recursion_result2}.

In summary, it holds that for any $(\rho,\theta)\in\Theta,s_t\in\calS,a_t\in\calA$, 
\begin{align}
\nabla_\theta Q_{\rho,\theta}(s_t,a_t) &= E_{\tau'\sim \pi_{\rho,\theta}}\Big[\sum_{k=t}^\infty \rho^{k-t}\nabla_\theta U_\theta(R(s_{k}',a_{k}'))\Big|s_t,a_t  \Big],\label{equ:gradient_Q_theta}\\
\nabla_\rho Q_{\rho,\theta}(s_t,a_t) &= E_{\tau'\sim \pi_{\rho,\theta}}\Big[\sum_{k=t+1}^\infty \rho^{k-t-1}V_{\rho,\theta}(s_{k}')\Big|s_t,a_t  \Big],\label{equ:gradient_Q_eho}\\
\nabla_\theta V_{\rho,\theta}(s_t) &=\E_{\tau'\sim \pi_{\rho,\theta}}\Big[\sum_{k=t}^\infty \rho^{k-t}\nabla_\theta U_\theta(R(s_{k}',a_{k}'))\Big|s_t \Big],\label{equ:gradient_V_theta}\\
\nabla_\rho V_{\rho,\theta}(s_t) &= \E_{\tau'\sim \pi_{\rho,\theta}}\Big[\sum_{k=t+1}^\infty \rho^{k-t-1}V_{\rho,\theta}(s_{k}')\Big|s_t \Big].\label{equ:gradient_V_rho}
\end{align}  

\noindent \underline{Step 2.} Next, we derive the gradients of the log-likelihood function. For any $(\rho,\theta)\in\Theta$,
\begin{align}
\mathcal{L}(\rho,\theta) &= \E_{\tau\sim\pi_{\Bar{\rho},\Bar{\theta}}}\Big[\sum_{t=0}^\infty \gamma^t \log \pi_{\rho,\theta}(a_t|s_t)\Big]
=\E_{\tau\sim\pi_{\Bar{\rho},\Bar{\theta}}}\Big[\sum_{t=0}^\infty \gamma^t \log \frac{e^{Q_{\rho,\theta}(s_t,a_t)}}{\SumInt_a e^{Q_{\rho,\theta}(s_t,a)}}\Big]\label{equ:likelihood_gradient_4.5}\\
&=\E_{\tau\sim\pi_{\Bar{\rho},\Bar{\theta}}}\Big[\sum_{t=0}^\infty \gamma^t \Big(Q_{\rho,\theta}(s_t,a_t) -V_{\rho,\theta}(s_t)\Big)\Big]\nonumber\\
&= \E_{\tau\sim\pi_{\Bar{\rho},\Bar{\theta}}}\Big[\sum_{t=0}^\infty \gamma^t U_\theta(R(s_t,a_t))\Big]-\E_{s_0\sim\mu(\cdot)}\Big[V_{\rho,\theta}(s_0)\Big]+(\rho-\gamma)\E_{\tau\sim\pi_{\Bar{\rho},\Bar{\theta}}}\Big[\sum_{t=1}^\infty \gamma^{t-1} V_{\rho,\theta}(s_{t})\Big],\label{equ:likelihood_gradient_1.5} 
\end{align}
where $\mu$ is the distribution of the initial state $s_0$. \eqref{equ:likelihood_gradient_4.5} holds by the optimality of the policy, and \eqref{equ:likelihood_gradient_1.5} holds by the soft Bellman equation. Taking the gradient of \eqref{equ:likelihood_gradient_1.5} with respect to $\theta$ gives
\begin{align}
\nabla_\theta \mathcal{L}(\rho,\theta)& =    \E_{\tau\sim\pi_{\Bar{\rho},\Bar{\theta}}}\Big[\sum_{t=0}^\infty \gamma^t \nabla_\theta U_\theta(R(s_t,a_t))\Big] - \E_{s_0\sim\mu(\cdot)}\Big[\nabla_\theta V_{\rho,\theta}(s_0)\Big]\nonumber\\
&\qquad+(\rho-\gamma)\E_{\tau\sim\pi_{\Bar{\rho},\Bar{\theta}}}\Big[\sum_{t=1}^\infty \gamma^{t-1} \nabla_\theta V_{\rho,\theta}(s_{t})\Big].\label{equ:likelihood_gradient_3} 
\end{align}
Combining \eqref{equ:likelihood_gradient_3} with \eqref{equ:gradient_V_theta} gives:
\begin{align}
\nabla_\theta \mathcal{L}(\rho,\theta) &= \E_{\tau\sim\pi_{\Bar{\rho},\Bar{\theta}}}\Big[\sum_{t=0}^\infty \gamma^t \nabla_\theta U_\theta(R(s_t,a_t)) \Big]-\E_{\tau\sim\pi_{\rho,\theta}}\Big[\sum_{t=0}^\infty {\rho}^t \nabla_\theta U_\theta(R(s_t,a_t)) \Big]\nonumber\\
&\qquad+ (\rho-\gamma) \E_{\tau\sim\pi_{\Bar{\rho},\Bar{\theta}}}\Big[\sum_{t=1}^\infty \gamma^{t-1}  
\E_{\tau'\sim\pi_{\rho,\theta}}\Big[\sum_{k=t}^\infty \rho^{k-t} \nabla_\theta U_\theta(R(s'_{k},a'_{k}))\,\Big|\,s_t\Big]\Big].\label{equ:likelihood_gradient_6}
\end{align}
Similarly, taking the gradient of \eqref{equ:likelihood_gradient_1.5} with respect to $\rho$ gives 
\begin{align}
\nabla_\rho \mathcal{L}(\rho,\theta)  
&= - \E_{s_0\sim\mu(\cdot)}\Big[\nabla_\rho V_{\rho,\theta}(s_0)\Big]
+ \E_{\tau\sim\pi_{\Bar{\rho},\Bar{\theta}}}\Big[\sum_{t=1}^\infty \gamma^{t-1} V_{\rho,\theta}(s_t)\Big]\nonumber\\
&\qquad +(\rho-\gamma)\E_{\tau\sim\pi_{\Bar{\rho},\Bar{\theta}}}\Big[\sum_{t=1}^\infty \gamma^{t-1} \nabla_\rho V_{\rho,\theta}(s_t)\Big].\label{equ:likelihood_gradient_5}
\end{align}
Combining \eqref{equ:likelihood_gradient_5} with \eqref{equ:gradient_V_rho} yields:
\begin{align}
\nabla_\rho \mathcal{L}(\rho,\theta)  &={\E_{\tau\sim\pi_{\Bar{\rho},\Bar{\theta}}}\Big[\sum_{t=1}^\infty \gamma^{t-1} V_{\rho,\theta}(s_t)\Big]} -\mathbb{E}_{\tau\sim \pi_{\rho,\theta}}\Big[\sum_{t=1}^\infty \rho^{t-1}V_{\rho,\theta}(s_{t}) \Big]  \nonumber\\
&\qquad+(\rho-\gamma)\E_{\tau\sim\pi_{\Bar{\rho},\Bar{\theta}}}\Big[\sum_{t=1}^\infty \gamma^{t-1} \E_{\tau'\sim \pi_{\rho,\theta}}\Big[\sum_{k=t+1}^\infty \rho^{k-t-1}V_{\rho,\theta}(s'_{k})\,\Big|\, s_t  \Big]\Big].\label{equ:likelihood_gradient_7}
\end{align}

\noindent \underline{Step 3.} Finally, when $(\rho,\theta) =(\Bar{\rho},\Bar{\theta})$, by \eqref{equ:likelihood_gradient_6} we have
\begin{eqnarray}
    \nabla_\theta \mathcal{L}(\rho,\theta)|_{(\rho,\theta) =(\Bar{\rho},\Bar{\theta})} &= \E_{\tau\sim\pi_{\Bar{\rho},\Bar{\theta}}}\Big[\sum_{t=0}^\infty \gamma^t \nabla_\theta U_\theta(R(s_t,a_t)) \Big]-\E_{\tau\sim\pi_{\Bar{\rho},\Bar{\theta}}}\Big[\sum_{t=0}^\infty {{\rho}^t} \nabla_\theta U_\theta(R(s_t,a_t)) \Big]\nonumber\\
&\qquad+ (\rho-\gamma) \E_{\tau\sim\pi_{\Bar{\rho},\Bar{\theta}}}\Big[\sum_{t=1}^\infty \gamma^{t-1}  
\sum_{k=t}^\infty \rho^{k-t} \nabla_\theta U_\theta(R(s_{k},a_{k}))\Big].\label{eq:grad0}
\end{eqnarray}
Note that for the last line of the above equation,
\begin{eqnarray}
&&\E_{\tau\sim\pi_{\Bar{\rho},\Bar{\theta}}}\Big[\sum_{t=1}^\infty \gamma^{t-1}  
\sum_{k=t}^\infty \rho^{k-t} \nabla_\theta U_\theta(R(s_{k},a_{k}))\Big]= \E_{\tau\sim\pi_{\Bar{\rho},\Bar{\theta}}}\Big[\sum_{t=1}^\infty \left(\frac{\gamma}{\rho}\right)^{t-1}  
\sum_{k=t}^\infty \rho^{k-1} \nabla_\theta U_\theta(R(s_{k},a_{k}))\Big]\nonumber\\
&=&   \E_{\tau\sim\pi_{\Bar{\rho},\Bar{\theta}}}\Big[\sum_{k=1}^\infty\sum_{t=1}^{k} \left(\frac{\gamma}{\rho}\right)^{t-1}  
 \rho^{k-1} \nabla_\theta U_\theta(R(s_{k},a_{k}))\Big]\label{eq:inter1}\\
&=&\E_{\tau\sim\pi_{\Bar{\rho},\Bar{\theta}}}\Big[\sum_{k=1}^\infty\frac{(\gamma/\rho)^{k}-1}{\gamma/\rho-1} 
 \rho^{k-1} \nabla_\theta U_\theta(R(s_{k},a_{k}))\Big] = \E_{\tau\sim\pi_{\Bar{\rho},\Bar{\theta}}}\Big[\sum_{k=1}^\infty\frac{\gamma^{k}-\rho^{k}}{\gamma-\rho} 
 \nabla_\theta U_\theta(R(s_{k},a_{k}))\Big],\label{eq:inter2}
\end{eqnarray}
where \eqref{eq:inter1} holds by changing the order of summations. Plugging \eqref{eq:inter2} back into \eqref{eq:grad0}, we have the desired result that $ \nabla_\theta \mathcal{L}(\rho,\theta)|_{(\rho,\theta) = (\Bar{\rho},\Bar{\theta})} = 0$.

Similarly we have $ \nabla_\rho \mathcal{L}(\rho,\theta)|_{(\rho,\theta) = (\Bar{\rho},\Bar{\theta})} = 0$.
\end{proof}

We next show results on the Hessian matrix.
\begin{theorem}[Landscape analysis]\label{lemma:landscape}
It holds that
\begin{eqnarray*}
\nabla^2_\theta \mathcal{L}(\rho,\theta){|_{(\rho,\theta) = (\Bar{\rho},\Bar{\theta})}} 
&=& - \E_{\tau\sim\pi_{\Bar{\rho},\Bar{\theta}}}\Big[\sum_{t=0}^\infty\gamma^{t}
\mathbb{V}_{a\sim\pi_{\Bar{\rho},\Bar{\theta}}(\cdot|s_{t})}\Big[\nabla_\theta Q_{\rho,\theta}(s_{t},a) \Big] \Big],\\
\nabla^2_\rho \mathcal{L}(\rho,\theta)  {|_{(\rho,\theta) = (\Bar{\rho},\Bar{\theta})}} 
&=& -  \E_{\tau\sim\pi_{\Bar{\rho},\Bar{\theta}}}\Big[\sum_{t=0}^\infty \gamma^t \mathbb{V}_{a\sim\pi_{\rho,\theta}(\cdot|s_{t})}\Big[\nabla_\rho Q_{\rho,\theta}(s_{t},a)\Big]\Big],\\
\nabla_\theta\nabla_\rho \mathcal{L}(\rho,\theta){|_{(\rho,\theta) = (\Bar{\rho},\Bar{\theta})}} 
&=& -\E_{\tau\sim\pi_{\Bar{\rho},\Bar{\theta}}}\Big[\sum_{t=0}^\infty\gamma^t{\rm \bf Cov}_{a\sim\pi_{\Bar{\rho},\Bar{\theta}}(\cdot|s_{t})}\Big[\nabla_\theta Q_{\rho,\theta}(s_{t},a),\nabla_\rho Q_{\rho,\theta}(s_{t},a)\Big]\Big],
\end{eqnarray*}
in which we define
\begin{align}
\mathbb{V}_{a\sim\pi_{\rho,\theta}(\cdot|s_t)}\Big[\nabla_\theta Q_{\rho,\theta}(s_t,a)\Big]
& :=\E_{a\sim\pi_{\rho,\theta}(\cdot|s_t)}\Big[\nabla_\theta Q_{\rho,\theta}(s_t,a)\nabla_\theta Q_{\rho,\theta}(s_t,a)^\top\Big]\nonumber\\
&\qquad- \E_{a\sim\pi_{\rho,\theta}(\cdot|s_t)}\Big[ \nabla_\theta Q_{\rho,\theta}(s_t,a)\Big]\E_{a\sim\pi_{\rho,\theta}(\cdot|s_t)}\Big[ \nabla_\theta Q_{\rho,\theta}(s_t,a)\Big]^\top,\label{equ:definition_of_variance}
\end{align}
and
\begin{align}
&{\rm \bf Cov}_{a\sim\pi_{\rho,\theta}(\cdot|s_{t+1})}\Big[\nabla_\theta Q_{\rho,\theta}(s_{t+1},a),\nabla_\rho Q_{\rho,\theta}(s_{t+1},a)\Big]\nonumber\\
&:=\E_{a\sim\pi_{\rho,\theta}(\cdot|s_t)}\Big[\nabla_\theta Q_{\rho,\theta}(s_t,a)\nabla_\rho Q_{\rho,\theta}(s_t,a)\Big]\nonumber
\\
&\qquad \qquad - \E_{a\sim\pi_{\rho,\theta}(\cdot|s_t)}\Big[ \nabla_\theta Q_{\rho,\theta}(s_t,a)\Big]\E_{a\sim\pi_{\rho,\theta}(\cdot|s_t)}\Big[ \nabla_\rho Q_{\rho,\theta}(s_t,a)\Big]\in\mathbb{R}^d.\label{equ:definition_of_covariance}
\end{align}

In addition,
\begin{eqnarray*}
\mathcal H(\Bar{\rho},\Bar{\theta})&&:=
\begin{pmatrix}
    \nabla^2_\theta \mathcal{L}(\rho,\theta) &  \nabla_\theta \nabla_\rho \mathcal{L}(\rho,\theta)\\
    \nabla_\theta \nabla_\rho \mathcal{L}(\rho,\theta)^\top &  \nabla^2_\rho \mathcal{L}(\rho,\theta)
\end{pmatrix}{\Bigg|_{(\rho,\theta) = (\Bar{\rho},\Bar{\theta})} }
\end{eqnarray*}
is negative semi-definite. 
\end{theorem}
Theorem \ref{lemma:landscape} suggests that the log-liklihood function $\mathcal{L}(\rho,\theta)$ is concave near the client's preference parameter $(\Bar{\rho},\Bar{\theta})$. As mentioned earlier, an interesting finding is that the negative semi-definite property of the Hessian  {\it does not rely on}  the choice $\gamma$, making the likelihood estimation method particularly suitable for inference problems.

\begin{proof}
Our proof consists of two parts. We first derive formulas for the second-order derivatives of $Q,V$ with respect to $\theta$ and $\rho$. Then we calculate the second-order derivatives of the log-likelihood function and study its Hessian matrix when $(\rho,\theta) = (\Bar{\rho},\Bar{\theta})$.
 
\noindent \underline{Step 1.} To begin with, for any $(\rho,\theta)\in\Theta$ by taking the derivative of \eqref{equ:likelihood_gradient_3.5}, we have:
\begin{align}
\nabla_\theta^2V_{\rho,\theta}(s_t)  
&= \nabla_\theta \Big(\SumInt_a \pi_{\rho,\theta}(a|s_t) \nabla_\theta Q_{\rho,\theta}(s_t,a)^\top\Big)\nonumber\\
&=\SumInt_a \nabla_\theta \Big( e^{Q_{\rho,\theta}(s_t,a)-V_{\rho,\theta}(s_t) } \Big)\nabla_\theta Q_{\rho,\theta}(s_t,a)^\top
+\SumInt_a \pi_{\rho,\theta}(a|s_t) \nabla_\theta^2 Q_{\rho,\theta}(s_t,a) \label{equ:second_order_derivative_1}\\
&= \E_{a\sim\pi_{\rho,\theta}(\cdot|s_t)}\Big[\nabla_\theta Q_{\rho,\theta}(s_t,a)\nabla_\theta Q_{\rho,\theta}(s_t,a)^\top\Big]\nonumber\\
&\qquad- \E_{a\sim\pi_{\rho,\theta}(\cdot|s_t)}\Big[ \nabla_\theta Q_{\rho,\theta}(s_t,a)\Big]\E_{a\sim\pi_{\rho,\theta}(\cdot|s_t)}\Big[ \nabla_\theta Q_{\rho,\theta}(s_t,a)\Big]^\top
+\E_{a\sim\pi_{\rho,\theta}(\cdot|s_t)}\Big[\nabla_\theta^2 Q_{\rho,\theta}(s_t,a)\Big]\label{equ:second_order_derivative_2}\\
&= \mathbb{V}_{a\sim\pi_{\rho,\theta}(\cdot|s_t)}\Big[\nabla_\theta Q_{\rho,\theta}(s_t,a)\Big]
+\E_{a\sim\pi_{\rho,\theta}(\cdot|s_t)}\Big[\nabla_\theta^2 Q_{\rho,\theta}(s_t,a)\Big],\label{equ:second_order_derivative_2.5}
\end{align}
where the covariance matrix $\mathbb{V}_{a\sim\pi_{\rho,\theta}(\cdot|s_t)}\Big[\nabla_\theta Q_{\rho,\theta}(s_t,a)\Big]$ is defined in \eqref{equ:definition_of_variance}.
In particular, \eqref{equ:second_order_derivative_1} holds because:
\begin{align*}
\pi_{\rho,\theta}(a|s_t) = \frac{e^{Q_{\rho,\theta}(s_t,a)}}{\SumInt_{a'}e^{Q_{\rho,\theta}(s,a')}} = e^{Q_{\rho,\theta}(s_t,a)-V_{\rho,\theta}(s_t) },
\end{align*}
and 
\eqref{equ:second_order_derivative_2} holds by \eqref{equ:likelihood_gradient_3.5}.
In addition, we have:
\begin{align}
\nabla^2_\theta Q_{\rho,\theta}(s_t,a_t) 
&= \nabla_\theta^2 U_\theta (R(s_t,a_t)) + \rho \E_{s_{t+1}\sim P(\cdot|s_t,a_t)}\Big[\nabla_\theta^2 V_{\rho,\theta}(s_{t+1})\Big]\nonumber\\
&= \nabla_\theta^2 U_\theta (R(s_t,a_t))
+\rho\E_{s_{t+1}\sim P(\cdot|s_t,a_t)}\Big[\mathbb{V}_{a_{t+1}\sim\pi_{\rho,\theta}(\cdot|s_{t+1})}\Big[\nabla_\theta Q_{\rho,\theta}(s_{t+1},a_{t+1})\Big]\Big]\nonumber\\
&\qquad
+\rho \E_{\tau\sim\pi_{\rho,\theta}}\Big[\nabla_\theta^2 Q_{\rho,\theta}(s_{t+1},a_{t+1})\Big|s_t,a_t\Big],\label{equ:second_order_derivative_3}
\end{align}
where \eqref{equ:second_order_derivative_3} holds by \eqref{equ:second_order_derivative_2.5}. Applying \eqref{equ:second_order_derivative_3} recursively yields:
\begin{align}
\nabla^2_\theta Q_{\rho,\theta}(s_t,a_t) 
& = \E_{\tau'\sim\pi_{\rho,\theta}}\Big[\sum_{k=t}^\infty \rho^{k-t}\nabla_\theta^2 U_\theta (R(s_{k}',a_{k}'))\Big|s_t,a_t \Big]\nonumber\\
&\qquad+ \E_{\tau'\sim\pi_{\rho,\theta}}\Big[\sum_{k=t}^\infty \rho^{k-t+1} \mathbb{V}_{a\sim\pi_{\rho,\theta}(\cdot|s_{k+1}')}\Big[\nabla_\theta Q_{\rho,\theta}(s_{k+1}',a)\Big]\Big|s_t,a_t\Big].\label{equ:second_order_derivative_2.6}
\end{align}
Similarly, for any $\rho\in(0,1)$, $\theta\in\Theta$, by taking the gradient  of \eqref{equ:likelihood_gradient_3.6} with respect to $\rho$, we have
\begin{align}
\nabla_\rho^2V_{\rho,\theta}(s_t)  
&= \nabla_\rho \Big(\SumInt_a \pi_{\rho,\theta}(a|s_t) \nabla_\rho Q_{\rho,\theta}(s_t,a)\Big)\nonumber\\
&=\SumInt_a \nabla_\rho \Big( e^{Q_{\rho,\theta}(s_t,a)-V_{\rho,\theta}(s_t) } \Big)\nabla_\rho Q_{\rho,\theta}(s_t,a)
+\SumInt_a \pi_{\rho,\theta}(a|s_t) \nabla_\rho^2 Q_{\rho,\theta}(s_t,a) \nonumber\\
&= \E_{a\sim\pi_{\rho,\theta}(\cdot|s_t)}\Big[(\nabla_\rho Q_{\rho,\theta}(s_t,a))^2\Big]
- \E_{a\sim\pi_{\rho,\theta}(\cdot|s_t)}\Big[ \nabla_\rho Q_{\rho,\theta}(s_t,a)\Big]^2
+\E_{a\sim\pi_{\rho,\theta}(\cdot|s_t)}\Big[\nabla_\rho^2 Q_{\rho,\theta}(s_t,a)\Big]\nonumber\\
&= \mathbb{V}_{a\sim\pi_{\rho,\theta}(\cdot|s_t)}\Big[\nabla_\rho Q_{\rho,\theta}(s_t,a)\Big]
+\E_{a\sim\pi_{\rho,\theta}(\cdot|s_t)}\Big[\nabla_\rho^2 Q_{\rho,\theta}(s_t,a)\Big].
\label{equ:second_order_derivative_2.7}
\end{align}
Similarly, by taking the gradient of \eqref{equ:likelihood_gradient_0.25},
\begin{align}
\nabla_\rho^2 Q_{\rho,\theta}(s_t,a_t)
&= \nabla_\rho\Big(\E_{s_{t+1}\sim P(\cdot|s_t,a_t)}\Big[ V_{\rho,\theta}(s_{t+1})\Big] 
+\rho \E_{s_{t+1}\sim P(\cdot|s_t,a_t)}\Big[ \nabla_\rho V_{\rho,\theta}(s_{t+1})\Big]\Big)\nonumber\\
&= 2\E_{s_{t+1}\sim P(\cdot|s_t,a_t)}\Big[ \nabla_\rho V_{\rho,\theta}(s_{t+1})\Big] 
+\rho \E_{s_{t+1}\sim P(\cdot|s_t,a_t)}\Big[ \nabla_\rho^2 V_{\rho,\theta}(s_{t+1})\Big]\nonumber\\
& = 2\E_{\tau\sim\pi_{\rho,\theta}}\Big[ \nabla_\rho Q_{\rho,\theta}(s_{t+1},a_{t+1})\Big|s_t,a_t\Big] 
+\rho \E_{s_{t+1}}\Big[\mathbb{V}_{a\sim\pi_{\rho,\theta}(\cdot|s_{t+1})}\Big[\nabla_\rho Q_{\rho,\theta}(s_{t+1},a)\Big]\Big] \nonumber\\
&\qquad+\rho \E_{a_{t+1}\sim\pi_{\rho,\theta}(\cdot|s_t)}\Big[\nabla_\rho^2 Q_{\rho,\theta}(s_{t+1},a_{t+1})\Big|s_t,a_t\Big].
\label{equ:second_order_derivative_6}
\end{align}
Applying \eqref{equ:second_order_derivative_6} recursively yields:
\begin{align}
\nabla_\rho^2 Q_{\rho,\theta}(s_t,a_t)  
&= 2\E_{\tau'\sim\pi_{\rho,\theta}}\Big[ \sum_{k=t}^\infty \rho^{k-t} \nabla_\rho Q_{\rho,\theta}(s_{k+1}',a_{k+1}')\Big|s_t,a_t\Big] \nonumber\\
&\qquad+\E_{\tau'\sim\pi_{\rho,\theta}}
\Big[\sum_{k=t}^\infty \rho^{k-t+1}\mathbb{V}_{a\sim\pi_{\rho,\theta}(\cdot|s_{k+1}')}\Big[\nabla_\rho Q_{\rho,\theta}(s_{k+1}',a)\Big]\Big|s_t,a_t\Big].\label{equ:second_order_derivative_2.8}
\end{align}
Furthermore, for any $(\rho,\theta)\in\Theta$, by taking the gradient of \eqref{equ:likelihood_gradient_3.6} with respect to $\theta$ , we have
\begin{align}
&\nabla_\theta\nabla_\rho V_{\rho,\theta}(s_t)  
= \nabla_\theta \Big(\SumInt_a \pi_{\rho,\theta}(a|s_t) \nabla_\rho Q_{\rho,\theta}(s_t,a)\Big)\nonumber\\
&=\SumInt_a \nabla_\theta \Big( e^{Q_{\rho,\theta}(s_t,a)-V_{\rho,\theta}(s_t) } \Big)\nabla_\rho Q_{\rho,\theta}(s_t,a)
+\SumInt_a \pi_{\rho,\theta}(a|s_t) \nabla_\theta\nabla_\rho Q_{\rho,\theta}(s_t,a) \nonumber\\
&= \E_{a\sim\pi_{\rho,\theta}(\cdot|s_t)}\Big[\nabla_\theta Q_{\rho,\theta}(s_t,a)\nabla_\rho Q_{\rho,\theta}(s_t,a)\Big]
- \E_{a\sim\pi_{\rho,\theta}(\cdot|s_t)}\Big[ \nabla_\theta Q_{\rho,\theta}(s_t,a)\Big]\E_{a\sim\pi_{\rho,\theta}(\cdot|s_t)}\Big[ \nabla_\rho Q_{\rho,\theta}(s_t,a)\Big]\nonumber\\
&\qquad+\E_{a\sim\pi_{\rho,\theta}(\cdot|s_t)}\Big[\nabla_\theta\nabla_\rho Q_{\rho,\theta}(s_t,a)\Big]\nonumber\\
& = {\rm \bf Cov}_{a\sim\pi_{\rho,\theta}(\cdot|s_{t+1})}\Big[\nabla_\theta Q_{\rho,\theta}(s_{t+1},a),\nabla_\rho Q_{\rho,\theta}(s_{t+1},a)\Big]\nonumber\\
&\qquad+\E_{a\sim\pi_{\rho,\theta}(\cdot|s_t)}\Big[\nabla_\theta\nabla_\rho Q_{\rho,\theta}(s_t,a)\Big],\label{equ:second_order_derivative_2.9}
\end{align}
where the ``covariance'' between $\nabla_\theta Q$ and $\nabla_\rho Q$ is defined in \eqref{equ:definition_of_covariance}.
Note that ${\rm \bf Cov}_{a\sim\pi_{\rho,\theta}(\cdot|s_{t+1})}\Big[\nabla_\theta Q_{\rho,\theta}(s_{t+1},a),\nabla_\rho Q_{\rho,\theta}(s_{t+1},a)\Big] \in \mathbb{R}^d$, as we have $\theta\in \mathbb{R}^d$ and $\rho \in \mathbb{R}$.

Lastly, by taking the gradient of \eqref{equ:likelihood_gradient_0.25}, we obtain
\begin{align*}
&\nabla_\theta\nabla_\rho Q_{\rho,\theta}(s_t,a_t)\\
&= \nabla_\theta\Big(\E_{s_{t+1}\sim P(\cdot|s_t,a_t)}\Big[ V_{\rho,\theta}(s_{t+1})\Big] 
+\rho \E_{s_{t+1}\sim P(\cdot|s_t,a_t)}\Big[ \nabla_\rho V_{\rho,\theta}(s_{t+1})\Big]\Big)\\
&= \E_{s_{t+1}\sim P(\cdot|s_t,a_t)}\Big[ \nabla_\theta V_{\rho,\theta}(s_{t+1})\Big] 
+\rho \E_{s_{t+1}\sim P(\cdot|s_t,a_t)}\Big[ \nabla_\theta\nabla_\rho V_{\rho,\theta}(s_{t+1})\Big]\\
& = \E_{s_{t+1}\sim P(\cdot|s_t,a_t)}\Big[ \nabla_\theta V_{\rho,\theta}(s_{t+1})\Big]
+ \rho \E_{s_{t+1}\sim P(\cdot|s_t,a_t)}\Big[ {\rm \bf Cov}_{a\sim\pi_{\rho,\theta}(\cdot|s_{t+1})}\Big[\nabla_\theta Q_{\rho,\theta}(s_{t+1},a),\nabla_\rho Q_{\rho,\theta}(s_{t+1},a)\Big]\Big]\\
&\qquad+ \rho \E_{\tau\sim\pi_{\rho,\theta}}\Big[\nabla_\theta\nabla_\rho Q_{\rho,\theta}(s_{t+1},a)\Big|s_t,a_t\Big].
\end{align*}
Applying the last equation recursively yields:
\begin{align}
\nabla_\theta\nabla_\rho Q_{\rho,\theta}(s_t,a_t)
&=\E_{\tau'\sim\pi_{\rho,\theta}}\Big[ \sum_{k=t}^\infty \rho^{k-t}\nabla_\theta V_{\rho,\theta}(s_{k+1}')\Big|s_t,a_t\Big]\nonumber\\
&\qquad+ \E_{\tau'\sim\pi_{\rho,\theta}}\Big[ \sum_{k=t}^\infty\rho^{k-t+1}{\rm \bf Cov}_{a\sim\pi_{\rho,\theta}(\cdot|s_{k+1})}\Big[\nabla_\theta Q_{\rho,\theta}(s_{k+1}',a),\nabla_\rho Q_{\rho,\theta}(s_{k+1}',a)\Big]\Big|s_t,a_t\Big].\label{equ:second_order_derivative_2.95}
\end{align}

In summary, by combining \eqref{equ:second_order_derivative_2.5}, \eqref{equ:second_order_derivative_2.6}, \eqref{equ:second_order_derivative_2.7}, \eqref{equ:second_order_derivative_2.8},\eqref{equ:second_order_derivative_2.9},
and \eqref{equ:second_order_derivative_2.95}, we have the following formulas of the second-order gradients of the value function: for any $(\rho,\theta)\in\Theta$, for any $(s_t,a_t)\in\calS\times\calA$,
\begin{align}
\nabla^2_\theta V_{\rho,\theta}(s_t) 
&= \E_{\tau'\sim\pi_{\rho,\theta}}\Big[\sum_{k=t}^\infty \rho^{k-t}\nabla_\theta^2 U_\theta (R(s_{k}',a_{k}')) \Big|s_t\Big]\nonumber\\
&\qquad+ \E_{\tau'\sim\pi_{\rho,\theta}}\Big[\sum_{k=t}^\infty \rho^{k-t} \mathbb{V}_{a\sim\pi_{\rho,\theta}(\cdot|s_{k}')}\Big[\nabla_\theta Q_{\rho,\theta}(s_{k}',a)\Big]\Big|s_t\Big],\label{eq:v_secondorder_1}\\
\nabla^2_\rho V_{\rho,\theta}(s_t) 
&= 2\E_{\tau'\sim\pi_{\rho,\theta}}\Big[ \sum_{k=t}^\infty \rho^{k-t} \nabla_\rho Q_{\rho,\theta}(s_{k+1}',a_{k+1})\Big|s_t\Big] \nonumber\\
&\qquad+\E_{\tau'\sim\pi_{\rho,\theta}}
\Big[\sum_{k=t}^\infty \rho^{k-t}\mathbb{V}_{a\sim\pi_{\rho,\theta}(\cdot|s_{k}')}\Big[\nabla_\rho Q_{\rho,\theta}(s_{k}',a)\Big]\Big|s_t\Big],\label{eq:v_secondorder_2}\\
\nabla_\theta\nabla_\rho V_{\rho,\theta}(s_t) 
&= \E_{\tau'\sim\pi_{\rho,\theta}}\Big[ \sum_{k=t}^\infty \rho^{k-t}\nabla_\theta V_{\rho,\theta}(s_{k+1}')\Big|s_t\Big]\nonumber\\
&\qquad+ \E_{\tau'\sim\pi_{\rho,\theta}}\Big[ \sum_{k=t}^\infty\rho^{k-t}{\rm \bf Cov}_{a\sim\pi_{\rho,\theta}(\cdot|s_{k}')}\Big[\nabla_\theta Q_{\rho,\theta}(s_{k}',a),\nabla_\rho Q_{\rho,\theta}(s_{k}',a)\Big]\Big|s_t\Big].\label{eq:v_secondorder_3}
\end{align}

\noindent \underline{Step 2.} Next, we calculate the derivatives of the log-likelihood function. By straight-forward calculations using \eqref{equ:likelihood_gradient_3} and combining with \eqref{eq:v_secondorder_1}, the second-order derivative of the log-likelihood function to $\theta$ satisfies:
\begin{align*}
\nabla^2_\theta \mathcal{L}(\rho,\theta) & =    \E_{\tau\sim\pi_{\Bar{\rho},\Bar{\theta}}}\Big[\sum_{t=0}^\infty \gamma^t \nabla^2_\theta U_\theta(R(s_t,a_t))\Big] - \E_{s_0\sim\mu(\cdot)}\Big[\nabla^2_\theta V_{\rho,\theta}(s_0)\Big]\nonumber\\
&\qquad+(\rho-\gamma)\E_{\tau\sim\pi_{\Bar{\rho},\Bar{\theta}}}\Big[\sum_{t=1}^\infty \gamma^{t-1} \nabla^2_\theta V_{\rho,\theta}(s_{t})\Big]\\
& =  \E_{\tau\sim\pi_{\Bar{\rho},\Bar{\theta}}}\Big[\sum_{t=0}^\infty \gamma^t \nabla^2_\theta U_\theta(R(s_t,a_t))\Big]
-\E_{\tau\sim\pi_{\rho,\theta}}\Big[\sum_{t=0}^\infty \rho^{t}\nabla_\theta^2 U_\theta (R(s_{t},a_{t})) \Big]\nonumber\\
&\qquad- \E_{\tau\sim\pi_{\rho,\theta}}\Big[\sum_{t=0}^\infty \rho^{t} \mathbb{V}_{a\sim\pi_{\rho,\theta}(\cdot|s_{t})}\Big[\nabla_\theta Q_{\rho,\theta}(s_{t},a)\Big]\Big]\\
&\qquad+(\rho-\gamma)\E_{\tau\sim\pi_{\Bar{\rho},\Bar{\theta}}}\Big[\sum_{t=1}^\infty \gamma^{t-1} \nabla^2_\theta V_{\rho,\theta}(s_{t})\Big].
\end{align*}
Note that when $(\rho,\theta)=(\Bar{\rho},\Bar{\theta})$,
\begin{eqnarray*}
\E_{\tau\sim\pi_{\Bar{\rho},\Bar{\theta}}}\Big[\sum_{t=1}^\infty \gamma^{t-1} \nabla^2_\theta V_{\rho,\theta}(s_{t})\Big]
&=& \E_{\tau\sim\pi_{\Bar{\rho},\Bar{\theta}}}\Big[\sum_{t=1}^\infty \gamma^{t-1} \sum_{k=t}^\infty \rho^{k-t}\nabla_\theta^2 U_\theta (R(s_{k},a_{k})) \Big]\\
 &&+  \E_{\tau\sim\pi_{\Bar{\rho},\Bar{\theta}}}\Big[\sum_{t=1}^\infty \gamma^{t-1}\sum_{k=t}^\infty \rho^{k-t} \mathbb{V}_{a\sim\pi_{\Bar{\rho},\Bar{\theta}}(\cdot|s_{k})}\Big[\nabla_\theta Q_{\rho,\theta}(s_{k},a) \Big]\Big]\\
 &=&  \E_{\tau\sim\pi_{\Bar{\rho},\Bar{\theta}}}\Big[\sum_{k=1}^\infty\frac{\gamma^{k}-\rho^{k}}{\gamma-\rho} 
\nabla_\theta^2 U_\theta (R(s_{k},a_{k})) \Big]\\
 &&+  \E_{\tau\sim\pi_{\Bar{\rho},\Bar{\theta}}}\Big[\sum_{k=1}^\infty\frac{\gamma^{k}-\rho^{k}}{\gamma-\rho} 
 \mathbb{V}_{a\sim\pi_{\Bar{\rho},\Bar{\theta}}(\cdot|s_{k})}\Big[\nabla_\theta Q_{\rho,\theta}(s_{k},a) \Big]\Big],
\end{eqnarray*}
where the first equality holds by \eqref{eq:v_secondorder_1} and the last equality holds by changing the order of summations. Plugging the above result to $\nabla^2_\theta \mathcal{L}(\rho,\theta)$, we obtain that, 
\begin{eqnarray*}
\nabla^2_\theta \mathcal{L}(\rho,\theta) |_{(\rho,\theta) = (\Bar{\rho},\Bar{\theta})}
& =&  \E_{\tau\sim\pi_{\Bar{\rho},\Bar{\theta}}}\Big[\sum_{t=0}^\infty \gamma^t \nabla^2_\theta U_\theta(R(s_t,a_t))\Big]
-\E_{\tau\sim\pi_{\Bar{\rho},\Bar{\theta}}}\Big[\sum_{t=0}^\infty \rho^{t}\nabla_\theta^2 U_\theta (R(s_{t},a_{t})) \Big]\nonumber\\
&\qquad& - \E_{\tau\sim\pi_{\Bar{\rho},\Bar{\theta}}}\Big[\sum_{t=0}^\infty \rho^{t} \mathbb{V}_{a\sim\pi_{\rho,\theta}(\cdot|s_{t})}\Big[\nabla_\theta Q_{\rho,\theta}(s_{t},a)\Big]\Big]\\
&\qquad&+\E_{\tau\sim\pi_{\Bar{\rho},\Bar{\theta}}}\Big[\sum_{t=1}^\infty\rho^{t}
\nabla_\theta^2 U_\theta (R(s_{t},a_{t})) \Big] - \E_{\tau\sim\pi_{\Bar{\rho},\Bar{\theta}}}\Big[\sum_{t=1}^\infty\gamma^{t}
\nabla_\theta^2 U_\theta (R(s_{t},a_{t})) \Big]\\
&\qquad&+\E_{\tau\sim\pi_{\Bar{\rho},\Bar{\theta}}}\Big[\sum_{t=1}^\infty\rho^{t}
\mathbb{V}_{a\sim\pi_{\Bar{\rho},\Bar{\theta}}(\cdot|s_{t})}\Big[\nabla_\theta Q_{\rho,\theta}(s_{t},a) \Big]\Big]\\
&\qquad&- \E_{\tau\sim\pi_{\Bar{\rho},\Bar{\theta}}}\Big[\sum_{t=1}^\infty\gamma^{t}
\mathbb{V}_{a\sim\pi_{\Bar{\rho},\Bar{\theta}}(\cdot|s_{t})}\Big[\nabla_\theta Q_{\rho,\theta}(s_{t},a) \Big] \Big]\\
&=& - \E_{\tau\sim\pi_{\Bar{\rho},\Bar{\theta}}}\Big[\sum_{{t=0}}^\infty\gamma^{t}
\mathbb{V}_{a\sim\pi_{\Bar{\rho},\Bar{\theta}}(\cdot|s_{t})}\Big[\nabla_\theta Q_{\rho,\theta}(s_{t},a) \Big] \Big].
\end{eqnarray*}
Similarly, using the result in \eqref{eq:v_secondorder_2} and \eqref{equ:likelihood_gradient_5},
\begin{eqnarray}
\nabla^2_\rho \mathcal{L}(\rho,\theta)  
&=& -2\E_{\tau\sim\pi_{\rho,\theta}}\Big[ \sum_{t=0}^\infty \rho^{t} \nabla_\rho Q_{\rho,\theta}(s_{t+1},a_{t+1})\Big] -\E_{\tau\sim\pi_{\rho,\theta}}
\Big[\sum_{t=0}^\infty \rho^{t}\mathbb{V}_{a\sim\pi_{\rho,\theta}(\cdot|s_{t})}\Big[\nabla_\rho Q_{\rho,\theta}(s_{t},a)\Big]\Big]\nonumber\\
&&+2 \E_{\tau\sim\pi_{\Bar{\rho},\Bar{\theta}}}\Big[\sum_{t=1}^\infty \gamma^{t-1} \nabla_\rho V_{\rho,\theta}(s_t)\Big]\nonumber\\
&& +(\rho-\gamma)\E_{\tau\sim\pi_{\Bar{\rho},\Bar{\theta}}}\Big[\sum_{t=1}^\infty \gamma^{t-1} \nabla^2_\rho V_{\rho,\theta}(s_t)\Big].\label{equ:likelihood_hessian3}
\end{eqnarray}

Note that when $(\rho,\theta) = (\Bar{\rho},\Bar{\theta})$,
\begin{eqnarray*}
\E_{\tau\sim\pi_{\Bar{\rho},\Bar{\theta}}}\Big[\sum_{t=1}^\infty \gamma^{t-1} \nabla^2_\rho V_{\rho,\theta}(s_t)\Big]
&=&2 \E_{\tau\sim\pi_{\Bar{\rho},\Bar{\theta}}}\Big[\sum_{t=1}^\infty \gamma^{t-1}\sum_{k=t}^\infty \rho^{k-t} \nabla_\rho Q_{\rho,\theta}(s_{k+1},a_{k+1})\Big]\\
&& + \E_{\tau\sim\pi_{\Bar{\rho},\Bar{\theta}}}\Big[\sum_{t=1}^\infty \gamma^{t-1} \sum_{k=t}^\infty \rho^{k-t} \mathbb{V}_{a\sim\pi_{\rho,\theta}(\cdot|s_{k})}\Big[\nabla_\rho Q_{\rho,\theta}(s_{k},a)\Big]\Big]\\
&=&2 \E_{\tau\sim\pi_{\Bar{\rho},\Bar{\theta}}}\Big[\sum_{k=1}^\infty\frac{\gamma^{k}-\rho^{k}}{\gamma-\rho} \nabla_\rho Q_{\rho,\theta}(s_{k+1},a_{k+1})\Big]\\
&& + \E_{\tau\sim\pi_{\Bar{\rho},\Bar{\theta}}}\Big[\sum_{k=1}^\infty\frac{\gamma^{k}-\rho^{k}}{\gamma-\rho} \mathbb{V}_{a\sim\pi_{\rho,\theta}(\cdot|s_{k})}\Big[\nabla_\rho Q_{\rho,\theta}(s_{k},a)\Big]\Big],
\end{eqnarray*}
where the last equality holds by changing the order of summations.
Plugging the above result to $\nabla^2_\rho \mathcal{L}(\rho,\theta)$, we have 

\begin{eqnarray*}
&\nabla^2_\rho& \mathcal{L}(\rho,\theta)  |_{(\rho,\theta) = (\Bar{\rho},\Bar{\theta})}\\
&=& -2\E_{\tau\sim\pi_{\rho,\theta}}\Big[ \sum_{t=0}^\infty \rho^{t} \nabla_\rho Q_{\rho,\theta}(s_{t+1},a_{t+1})\Big] -\E_{\tau\sim\pi_{\rho,\theta}}
\Big[\sum_{t=0}^\infty \rho^{t}\mathbb{V}_{a\sim\pi_{\rho,\theta}(\cdot|s_{t})}\Big[\nabla_\rho Q_{\rho,\theta}(s_{t},a)\Big]\Big]\nonumber\\
&&+2 \E_{\tau\sim\pi_{\Bar{\rho},\Bar{\theta}}}\Big[\sum_{t=1}^\infty \gamma^{t-1} \nabla_\rho V_{\rho,\theta}(s_t)\Big]\nonumber\\
&& + 2 \E_{\tau\sim\pi_{\Bar{\rho},\Bar{\theta}}}\Big[\sum_{t=1}^\infty \rho^{t} \nabla_\rho Q_{\rho,\theta}(s_{t+1},a_{t+1})\Big]-2 \E_{\tau\sim\pi_{\Bar{\rho},\Bar{\theta}}}\Big[\sum_{t=1}^\infty \gamma^{t} \nabla_\rho Q_{\rho,\theta}(s_{t+1},a_{t+1})\Big] \nonumber\\
&& + \E_{\tau\sim\pi_{\Bar{\rho},\Bar{\theta}}}\Big[\sum_{t=1}^\infty \rho^t \mathbb{V}_{a\sim\pi_{\Bar{\rho},\Bar{\theta}}(\cdot|s_{t})}\Big[\nabla_\rho Q_{\rho,\theta}(s_{t},a)\Big]\Big] -  \E_{\tau\sim\pi_{\Bar{\rho},\Bar{\theta}}}\Big[\sum_{t=1}^\infty \gamma^t \mathbb{V}_{a\sim\pi_{\Bar{\rho},\Bar{\theta}}(\cdot|s_{t})}\Big[\nabla_\rho Q_{\rho,\theta}(s_{t},a)\Big]\Big] \\
&=& {2\, \E_{\tau\sim\pi_{\Bar{\rho},\Bar{\theta}}}\Big[\sum_{t=1}^\infty \gamma^{t-1} \nabla_\rho V_{\rho,\theta}(s_t)\Big] - 2\, \E_{\tau\sim\pi_{\Bar{\rho},\Bar{\theta}}}\Big[\sum_{t=1}^\infty \gamma^{t} \nabla_\rho Q_{\rho,\theta}(s_{t+1},a_{t+1})\Big]}
\\
&&-  \E_{\tau\sim\pi_{\Bar{\rho},\Bar{\theta}}}\Big[\sum_{t=1}^\infty \gamma^t \mathbb{V}_{a\sim\pi_{\Bar{\rho},\Bar{\theta}}(\cdot|s_{t})}\Big[\nabla_\rho Q_{\rho,\theta}(s_{t},a)\Big]\Big] \\
&&{-2 \E_{\tau\sim\pi_{\Bar{\rho},\Bar{\theta}}}\Big[\nabla_\rho Q_{\rho,\theta}(s_1,a_1)\Big] 
-\E_{\tau\sim\pi_{\Bar{\rho},\Bar{\theta}}}\Big[\mathbb{V}_{a\sim\pi_{\rho,\theta}(\cdot|s_{0})}\Big[\nabla_\rho Q_{\rho,\theta}(s_{0},a)\Big]\Big]}\\
&=& {-  \E_{\tau\sim\pi_{\Bar{\rho},\Bar{\theta}}}\Big[\sum_{t=0}^\infty \gamma^t \mathbb{V}_{a\sim\pi_{\Bar{\rho},\Bar{\theta}}(\cdot|s_{t})}\Big[\nabla_\rho Q_{\rho,\theta}(s_{t},a)\Big]\Big],}
\end{eqnarray*}
where the last equality holds by the expressions of $\nabla_\rho V_{\rho,\theta}$ and $\nabla_\rho Q_{\rho,\theta}$ in \eqref{equ:likelihood_gradient_3.5} and \eqref{equ:likelihood_gradient_3.6}.

Similarly,
\begin{eqnarray}
\nabla_\theta\nabla_\rho \mathcal{L}(\rho,\theta)& =&    - \E_{s_0\sim\mu(\cdot)}\Big[\nabla_\theta \nabla_\rho \, V_{\rho,\theta}(s_0)\Big]\nonumber+\E_{\tau\sim\pi_{\Bar{\rho},\Bar{\theta}}}\Big[\sum_{t=1}^\infty \gamma^{t-1} \nabla_\theta V_{\rho,\theta}(s_{t})\Big]\\
&&\qquad +(\rho-\gamma)\E_{\tau\sim\pi_{\Bar{\rho},\Bar{\theta}}}\Big[\sum_{t=1}^\infty \gamma^{t-1} \nabla_\theta\nabla_\rho V_{\rho,\theta}(s_{t})\Big].\label{eq:second_cross}
\end{eqnarray}
Note that when $(\rho,\theta)=(\Bar{\rho},\Bar{\theta})$,
\begin{eqnarray*} &&\E_{\tau\sim\pi_{\Bar{\rho},\Bar{\theta}}}\Big[\sum_{t=1}^\infty \gamma^{t-1} \nabla_\theta\nabla_\rho V_{\rho,\theta}(s_{t})\Big]\\
&=&\E_{\tau\sim\pi_{\Bar{\rho},\Bar{\theta}}}\Big[\sum_{t=1}^\infty \gamma^{t-1}\sum_{k=t}^\infty \rho^{k-t}\nabla_\theta V_{\rho,\theta}(s_{k+1})\Big] \\
&&+ \E_{\tau\sim\pi_{\Bar{\rho},\Bar{\theta}}}\Big[\sum_{t=1}^\infty \gamma^{t-1}\sum_{k=t}^\infty \rho^{k-t}{\rm \bf Cov}_{a\sim\pi_{\Bar{\rho},\Bar{\theta}}(\cdot|s_{k})}\Big[\nabla_\theta Q_{\rho,\theta}(s_{k},a),\nabla_\rho Q_{\rho,\theta}(s_{k},a)\Big]\Big]\\
&=&\E_{\tau\sim\pi_{\Bar{\rho},\Bar{\theta}}}\Big[\sum_{k=1}^\infty\frac{\gamma^{k}-\rho^{k}}{\gamma-\rho}\nabla_\theta V_{\rho,\theta}(s_{k+1})\Big] \\
&&+ \E_{\tau\sim\pi_{\Bar{\rho},\Bar{\theta}}}\Big[\sum_{k=1}^\infty\frac{\gamma^{k}-\rho^{k}}{\gamma-\rho}{\rm \bf Cov}_{a\sim\pi_{\Bar{\rho},\Bar{\theta}}(\cdot|s_{k})}\Big[\nabla_\theta Q_{\rho,\theta}(s_{k},a),\nabla_\rho Q_{\rho,\theta}(s_{k},a)\Big]\Big].
\end{eqnarray*}
The last equality holds by changing the order of summations. Plugging the last equality back to \eqref{eq:second_cross} and applying \eqref{eq:v_secondorder_3} for $t=0$, we have 
\begin{eqnarray}
&\nabla_\theta\nabla_\rho& \mathcal{L}(\rho,\theta)|_{(\rho,\theta) = (\Bar{\rho},\Bar{\theta})}\nonumber\\ 
&=&   - \E_{\tau\sim\pi_{\Bar{\rho},\Bar{\theta}}}\Big[\sum_{t=0}^\infty \rho^{t}\nabla_\theta V_{\rho,\theta}(s_{t+1})\Big]
-\E_{\tau\sim\pi_{\Bar{\rho},\Bar{\theta}}}\Big[\sum_{t=0}^\infty\rho^{t}{\rm \bf Cov}_{a\sim\pi_{\rho,\theta}(\cdot|s_{t})}\Big[\nabla_\theta Q_{\rho,\theta}(s_{t},a),\nabla_\rho Q_{\rho,\theta}(s_{t},a)\Big]\Big]\nonumber\\
&&+\E_{\tau\sim\pi_{\Bar{\rho},\Bar{\theta}}}\Big[\sum_{t=1}^\infty \gamma^{t-1} \nabla_\theta V_{\rho,\theta}(s_{t})\Big]\nonumber\\
&&+\E_{\tau\sim\pi_{\Bar{\rho},\Bar{\theta}}}\Big[\sum_{t=1}^\infty\rho^t\nabla_\theta V_{\rho,\theta}(s_{t+1})\Big]-\E_{\tau\sim\pi_{\Bar{\rho},\Bar{\theta}}}\Big[\sum_{t=1}^\infty\gamma^t\nabla_\theta V_{\rho,\theta}(s_{t+1})\Big]\nonumber \\
&&+ \E_{\tau\sim\pi_{\Bar{\rho},\Bar{\theta}}}\Big[\sum_{t=1}^\infty\rho^t{\rm \bf Cov}_{a\sim\pi_{\Bar{\rho},\Bar{\theta}}(\cdot|s_{t})}\Big[\nabla_\theta Q_{\rho,\theta}(s_{t},a),\nabla_\rho Q_{\rho,\theta}(s_{t},a)\Big]\Big]\nonumber\\
&& -\E_{\tau\sim\pi_{\Bar{\rho},\Bar{\theta}}}\Big[\sum_{t=1}^\infty\gamma^t{\rm \bf Cov}_{a\sim\pi_{\Bar{\rho},\Bar{\theta}}(\cdot|s_{t})}\Big[\nabla_\theta Q_{\rho,\theta}(s_{t},a),\nabla_\rho Q_{\rho,\theta}(s_{t},a)\Big]\Big]\nonumber\\
&=& -\E_{\tau\sim\pi_{\Bar{\rho},\Bar{\theta}}}\Big[\sum_{t=1}^\infty\gamma^t{\rm \bf Cov}_{a\sim\pi_{\Bar{\rho},\Bar{\theta}}(\cdot|s_{t})}\Big[\nabla_\theta Q_{\rho,\theta}(s_{t},a),\nabla_\rho Q_{\rho,\theta}(s_{t},a)\Big]\Big] \nonumber\\
&&{-\E_{\tau\sim\pi_{\Bar{\rho},\Bar{\theta}}}\Big[{\rm \bf Cov}_{a\sim\pi_{\rho,\theta}(\cdot|s_{0})}\Big[\nabla_\theta Q_{\rho,\theta}(s_{0},a),\nabla_\rho Q_{\rho,\theta}(s_{0},a)\Big]\Big]}\label{equ:landscape_1}\\
&=& {-\E_{\tau\sim\pi_{\Bar{\rho},\Bar{\theta}}}\Big[\sum_{t=0}^\infty\gamma^t{\rm \bf Cov}_{a\sim\pi_{\Bar{\rho},\Bar{\theta}}(\cdot|s_{t})}\Big[\nabla_\theta Q_{\rho,\theta}(s_{t},a),\nabla_\rho Q_{\rho,\theta}(s_{t},a)\Big]\Big]},\nonumber
\end{eqnarray}
where \eqref{equ:landscape_1} holds by the fact that
\begin{align*}
&\E_{\tau\sim\pi_{\Bar{\rho},\Bar{\theta}}}\Big[\sum_{t=1}^\infty \gamma^{t-1} \nabla_\theta V_{\rho,\theta}(s_{t})\Big]
-\E_{\tau\sim\pi_{\Bar{\rho},\Bar{\theta}}}\Big[\sum_{t=1}^\infty\gamma^t\nabla_\theta V_{\rho,\theta}(s_{t+1})\Big]
=\E_{\tau\sim\pi_{\Bar{\rho},\Bar{\theta}}}\Big[  \nabla_\theta V_{\rho,\theta}(s_{1})\Big].
\end{align*}
To summarize,
\begin{eqnarray*}
\mathcal H&&:=
\begin{pmatrix}
    \nabla^2_\theta \mathcal{L}(\rho,\theta) &  \nabla_\theta \nabla_\rho \mathcal{L}(\rho,\theta)\\
    \nabla_\theta \nabla_\rho \mathcal{L}(\rho,\theta)^\top &  \nabla^2_\rho \mathcal{L}(\rho,\theta)
\end{pmatrix}{\Bigg|_{(\rho,\theta) = (\Bar{\rho},\Bar{\theta})} }\\
&&=-\E_{\tau\sim\pi_{\Bar{\rho},\Bar{\theta}}}\Big[\sum_{t=0}^\infty\gamma^t
{\mathbb{V}}_{a\sim\pi_{\Bar{\rho},\Bar{\theta}}(\cdot|s_{t})}\Big[
\begin{pmatrix}
\nabla_\theta Q_{\rho,\theta}(s_{t},a)\\
\nabla_\rho Q_{\rho,\theta}(s_{t},a)
\end{pmatrix}\Big]\Big].
\end{eqnarray*}
Therefore, $\mathcal H$ is negative-semi definite by the definition of the covariance notation ${\mathbb{V}}$ in \eqref{equ:definition_of_variance}.

\end{proof}

\subsection{Algorithm Design and Implementation}
Motivated by the landscape analysis in Section \ref{sec:mle}, we design an algorithm that iteratively updates $\rho$ and $\theta$ to maximize the likelihood function; see Algorithm \ref{alg:ML}. 
At each iteration $k$, the value function $V_{\rho^k,\theta^k}$ is first computed by the soft Q iteration (see e.g. \cite{reddy2019sqil})  in lines \ref{line:Qsoft}-\ref{line:Vsoft}, and the parameters $\rho^k,\theta^k$ are then updated in line \ref{line:gradient2} using the gradient computed in line \ref{line:gradient}.

\begin{algorithm}[H]
\caption{Maximum likelihood update}
\label{alg:ML}
\begin{algorithmic}[1]
\State Initialize $\rho^0,\theta^0$.
\For{$k=1,2,\cdots, K$}
\State\label{line:Qsoft} Set $Q_{\rho^k,\theta^k}^{0}(s,a)=\frac{1}{1-\rho^k} $ for all $(s,a)\in\calS\times\calA$,
and $V_{\rho^k,\theta^k}^{0}(s)=\frac{1}{1-\rho^k} $ for all $s\in\calS$.
\For{$i=1,2,\cdots, I$}
\For{$(s,a)\in\calS\times\calA$}
\begin{align*}
Q_{\rho^k,\theta^k}^{i}(s,a) = U_{\theta^k}(R(s,a)) +\rho^k\, \E_{s'\sim P(\cdot|s,a)}\Big[V_{\rho^k,\theta^k}^{i}(s')\Big].
\end{align*}
\EndFor
\State\label{line:Vsoft} Compute $V_{\rho^k,\theta^k}^{I}$ using the soft Bellman equation \eqref{equ:soft_Bellman_equation}.
\EndFor
\State {With the value of $V^{I}_{\rho^k,\theta^k}$, compute $\nabla \mathcal{L}(\rho^k,\theta^k)$ using \eqref{equ:likelihood_gradient_6} and \eqref{equ:likelihood_gradient_7}.}
\label{line:gradient}
\State \label{line:gradient2} Update $(\rho^{k+1},\theta^{k+1}) = (\rho^{k},\theta^{k}) + \zeta^k \nabla \mathcal{L}(\rho^k,\theta^k)$.
\EndFor
\end{algorithmic}
\end{algorithm}

\paragraph{Numerical example one: Merton's problem}

We implement the discrete-time version of Merton's problem introduced in Section \ref{subsubsec:gen-discnt-inf}. The price of the bond follows $S_{t+1}^0=S_t^0+r\,\Delta$ and the price of the stock follows $S_{t+1}-S_t = S_t(\nu\Delta+\sigma \sqrt{\Delta} B_t)$, where $B_t$ are iid sampled from $\mathcal{N}(0,1)$.
Denote $(\alpha_t,c_t)\in \mathcal{A}:=[0,1]\times[0,2]$ as the pair of the consumption-allocation policy at time $t$, then the wealth process follows:
\begin{eqnarray}
X_{t+1}-X_t = \Big[ X_t (\alpha_t \nu + (1-
    \alpha_t)r)-c_t\Big] \Delta + X_t\alpha_t \sigma \sqrt{\Delta}B_t.
\end{eqnarray}
The client provides a time-homogeneous policy $\pi_{\Bar{\rho},\Bar{\theta}}\in \mathcal{P}(\mathcal{A})$ to the inference agent, which solves
\begin{eqnarray}
 \sup_{\pi}   \mathbb{E}\Big[\sum_{t=1}^\infty (\Bar{\rho})^t\Big( 1-\exp(-\Bar{\theta}\,c_t)\Big)+\mathcal{H}\big(\pi(\cdot|X_t)\big)\Big)\Big]
\end{eqnarray}
with $c_t = c(X_t)$ and $\alpha_t = \alpha(X_t)$. Here both $\Bar{\rho}$ and $\Bar{\theta}$ are unknown.

\begin{figure}[H]
    \centering
    \includegraphics[width=0.8\textwidth]{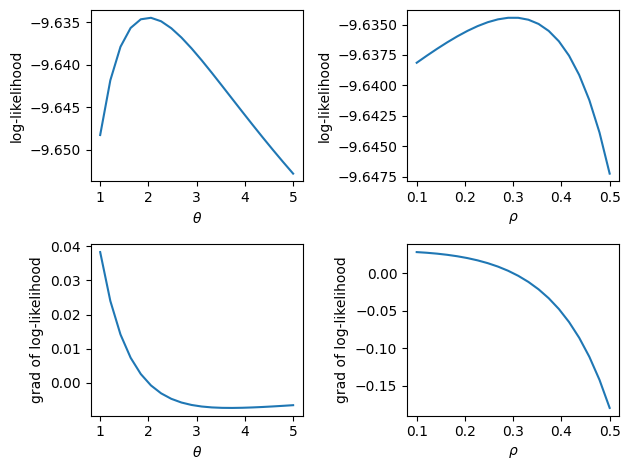}
\caption{Visualization of the log-likelihood function and its gradients ({\bf Left columns}: visualization with respect to $\theta$ (under $\rho = \Bar{\rho}$). {\bf Right columns}: visualization with respect to $\rho$ (under $\theta = \Bar{\theta}$)).}
        \label{fig:mertons_likelihood}
\end{figure}

In the experiment, we set 
$\Bar{\rho} = 0.3$, $\Bar{\theta} = 2$, $\gamma = 0.6$, $r=1.05$, $\Delta=1$, $\nu=1.06$, and $\sigma=0.05$.  We discretize and truncate the state space of the wealth process as $\calS = \{0.13,  0.39,  \dots,  2.23,  2.5\}$, with evenly distanced values such that $|\calS| = 10$. In addition, we discretize the joint space of the allocation and consumption processes as $\calA = \{0.1, 0.11, \dots, 0.98, 1\}\times\{0    , 0.22,\dots , 1.77, 2\}$, with evenly distanced values such that $|\calA|=50$.

We visualize the log-likelihood function and its gradient in Figure \ref{fig:mertons_likelihood}. One can see that the likelihood function is locally concave in $\theta$ and $\rho$  around $(\Bar{\theta},\Bar{\rho})$ in a sufficiently large area, enabling us to find the true parameters by Algorithm \ref{alg:ML} under fast convergence rate.

When implementing Algorithm \ref{alg:ML}, we initialize the parameters randomly with $\theta^0$ sampled uniformly from $[0,1]$ and $\rho^0$ sampled uniformly from $[0.1,0.2]$. We set the learning rate as $\zeta^k = \frac{1000}{k}$ and the total steps of the soft Q update as  $I=100$. As shown in Figure \ref{fig:mertons_convergence}, both $\theta$ and $\rho$ converge to the ground-truth value within 100 iterations.

Additionally, we analyze the behaviors of the client under different $\Bar{\rho}$ values. Figure \ref{fig:mertons_heatmap_comparison} suggests that the client opts for an overall higher consumption when $\Bar{\rho} = 0.1$ and an overall lower consumption when $\Bar{\rho} = 0.75$, indicating a bigger emphasis on deferred outcomes for the latter case.

\begin{figure}[H]
    \centering
    \includegraphics[width=0.5\textwidth]{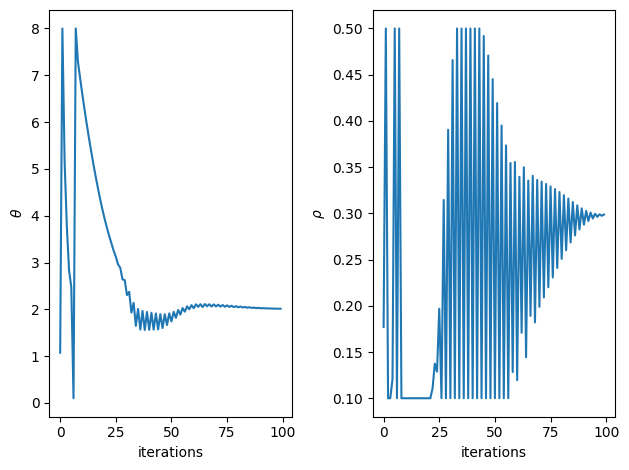}
    \caption{The convergence result of Algorithm \ref{alg:ML}. The left plot shows the value of $\theta$ at each iteration, while the right plot displays the values for $\rho$.}
    \label{fig:mertons_convergence}
\end{figure}

\begin{figure}[H]
    \centering
    \includegraphics[width=0.8\textwidth]{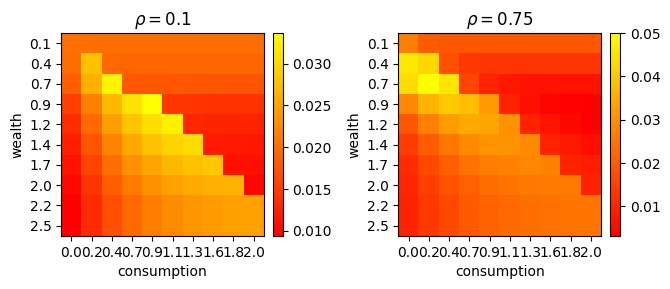}
    \caption{Visualization of the client's consumption policy. The left plot illustrates consumption at various wealth levels under $\Bar{\rho} = 0.1$, while the right plot corresponds to $\Bar{\rho} = 0.75$.}
    \label{fig:mertons_heatmap_comparison}
\end{figure}

\paragraph{Numerical example two: Investment under unhedgeable risk.} 
We consider a more complex investment problem, where the price of the primitive asset is modeled as a diffusion process whose coefficients evolve according to a correlated diffusive factor \cite{zariphopoulou2001solution}. The price of the bond follows the same dynamics as in Example One: $$S_{t+1}^0=S_t^0+r\,\Delta.$$ On the other hand, the price of the stock follows $$S_{t+1}-S_t = S_t(\nu(Y_t,t)\Delta+\sigma(Y_t,t) \sqrt{\Delta} B^1_t),$$ 
with $Y_t$ the ``stochastic factor model'' and it is assumed to satisfy $$Y_{t+1}-Y_t = b(Y_t,t)\Delta + d(Y_t,t)\sqrt{\Delta} B^1_t.$$ Here $B^1_t$  and  $B^2_t$ are iid sampled from $\mathcal{N}(0,1)$. We assume the correlation between $B_t^1$ and $B_t^2$ is $\eta\in(0,1)$.

Consider a problem with only investment and no consumption. Then the wealth process follows:
\begin{eqnarray}
    X_{t+1}-X_t = \Big[ X_t (\alpha_t \,\nu(t,Y_t) + (1-
    \alpha_t)r)\Big] \Delta + X_t\alpha_t \sigma (t,Y_t)\,\sqrt{\Delta}B^1_t,
\end{eqnarray}
under the investment strategy $\alpha_t\in \mathcal{A}=[0,1]$. The client provides a time-homogeneous policy $\pi_{\Bar{\rho},\Bar{\theta}}(x,y)\in \mathcal{P}(\mathcal{A})$ to the inference agent, which solves
\begin{eqnarray}
\label{equ:unhegeable}
 \sup_{\pi}   \mathbb{E}\Big[\sum_{t=1}^\infty (\Bar{\rho})^t \frac{1}{\Bar{\theta}_1 \Bar{\theta}_2}(X_t)^{\Bar{\theta}_1} (Y_t)^{\Bar{\theta}_2} + \mathcal{H}\Big(\pi(\cdot|(X_t,Y_t))\Big)\Big]
\end{eqnarray}
for some  $\Bar{\rho}>0$ and $\Bar{\theta}_1,\Bar{\theta}_2\in(0,1)$ that  are unknown to the inference agent.

\begin{figure}[H]
    \centering
    \includegraphics[width=0.8\textwidth]{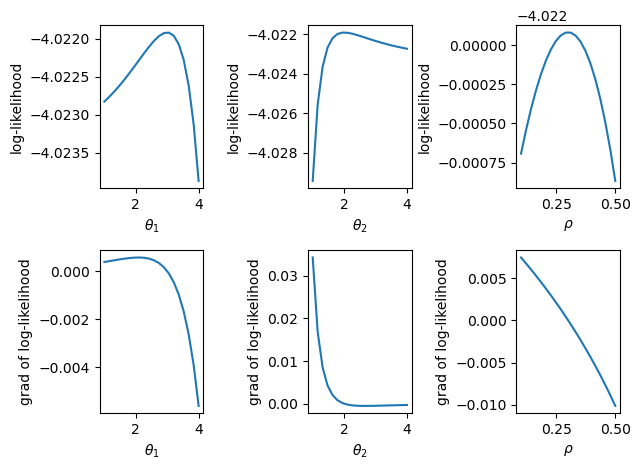}
\caption{Visualization of the log-likelihood function and its gradients ({\bf Left columns}: visualization with respect to $\theta_1$ (under $\theta_2 = \Bar{\theta}_2$ and $\rho = \Bar{\rho}$).
{\bf Middle columns}: visualization with respect to $\theta_2$ (under $\theta_1 = \Bar{\theta}_1$ and $\rho = \Bar{\rho}$).
{\bf Right columns}: visualization with respect to $\rho$ (under $\theta_1 = \Bar{\theta}_1$ and $\theta_2 = \Bar{\theta}_2$).)}
\label{fig:unhegeable_risk_likelihood}
\end{figure} 

In the experiment, we discretize and truncate the state space for the wealth process and the stochastic factor model as $\calS=\{0.1, 0.7, 1.3, 1.9, 2.5\}\times \{0.1 , 0.32, 0.55 , 0.77, 1\}$,  with evenly distanced values such that $|\calS|=25$. In addition, we discretize the action space of the allocation process as $\calA= \{0.1  , 0.32, 0.55 , 0.77, 1 \}$, with evenly distanced values such that  $|\calA|=5$.
We set $r = 1.05$, $\Delta = 1$, 
$\Bar{\theta_1} = 3$, $\Bar{\theta}_2 = 2$, $\Bar{\rho} = 0.3$, and $\gamma = 0.6$. For the drift and diffusion terms, we set $b(y,t) = -0.6y+0.2$, 
$d(y,t)=0.3 y + 0.3$, $\nu(t,y)=y$, and $\sigma(t,y)=0.5y+0.3$.

As shown in Figure \ref{fig:unhegeable_risk_likelihood}, 
we visualize the log-likelihood function and its gradient. One can see that the likelihood function is locally concave in $\theta$ and $\rho$ in an area around $(\Bar{\theta},\Bar{\rho})$.

When implementing Algorithm \ref{alg:ML}, we initialize the parameters randomly with $\theta_1^0$, $\theta_2^0$ sampled uniformly from $[1,2]$ and $\rho^0$ sampled uniformly from $[0.1,0.2]$. We set the learning rate as $\zeta^k = \frac{1000}{\sqrt{k}}$ and the total steps of the soft Q update as $I=100$. As shown in Figure \ref{fig:unhegeable_risk_convergence}, both $\theta_1,\theta_2$ and $\rho$ converge to the ground-truth values within 1500 iterations.

\begin{figure}[H]
    \centering
    \includegraphics[width=0.6\textwidth]{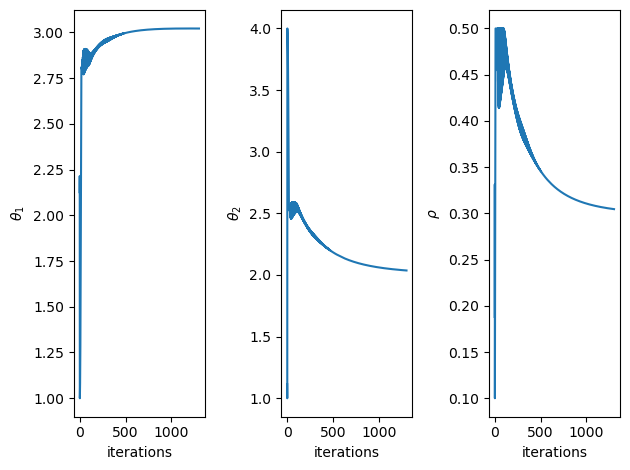}
    \caption{The convergence result of Algorithm \ref{alg:ML}. The left plot shows the value of $\theta_1$ at each iteration, the middle plot is for $\theta_2$, and the right plot is for $\rho$.}
    \label{fig:unhegeable_risk_convergence}
\end{figure}

Furthermore, Figure \ref{fig:unhegeable_risk_heatmap} illustrates the client's investment allocation policy $\alpha$ across various wealth levels (under fixed factor  value $1$), considering $\Bar{\rho} = 0.1$ and $\Bar{\rho} = 0.75$. The influence of $\Bar{\rho}$ on the investment decisions in this example is less pronounced compared to Merton's problem. This difference arises because, in Merton's problem, the client confronts a trade-off between higher consumption for instantaneous rewards and lower consumption for better future rewards. Conversely, the client addressing \eqref{equ:unhegeable} strives for a higher $X_t$ regardless of her $\Bar{\rho}$. Our algorithm consistently finds the optimal parameters, although the convergence speed here is slower compared to that for Merton's problem due to the above-mentioned reasons.

\begin{figure}[H]
    \centering
    \includegraphics[width=0.8\textwidth]{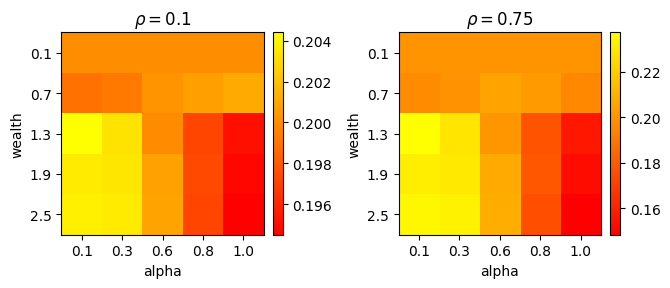}
    \caption{Visualization of the client's allocation policy (under fixed factor value $1$). The left plot illustrates her allocation at various wealth levels with $\Bar{\rho} = 0.1$, while the right plot is for $\Bar{\rho} = 0.75$.}
\label{fig:unhegeable_risk_heatmap}
\end{figure}

\bibliography{references}  

\begin{thebibliography}{66}
\providecommand{\natexlab}[1]{#1}
\providecommand{\url}[1]{\texttt{#1}}
\expandafter\ifx\csname urlstyle\endcsname\relax
  \providecommand{\doi}[1]{doi: #1}\else
  \providecommand{\doi}{doi: \begingroup \urlstyle{rm}\Url}\fi

\bibitem[Abbeel and Ng(2004)]{abbeel2004apprenticeship}
P.~Abbeel and A.~Y. Ng.
\newblock Apprenticeship learning via inverse reinforcement learning.
\newblock In \emph{Proceedings of the twenty-first international conference on
  Machine learning}, page~1, 2004.

\bibitem[Alsabah et~al.(2021)Alsabah, Capponi, Ruiz~Lacedelli, and
  Stern]{alsabah2021robo}
H.~Alsabah, A.~Capponi, O.~Ruiz~Lacedelli, and M.~Stern.
\newblock Robo-advising: Learning investors’ risk preferences via portfolio
  choices.
\newblock \emph{Journal of Financial Econometrics}, 19\penalty0 (2):\penalty0
  369--392, 2021.

\bibitem[Amin and Singh(2016)]{amin2016towards}
K.~Amin and S.~Singh.
\newblock Towards resolving unidentifiability in inverse reinforcement
  learning.
\newblock \emph{arXiv preprint arXiv:1601.06569}, 2016.

\bibitem[Amin et~al.(2017)Amin, Jiang, and Singh]{amin2017repeated}
K.~Amin, N.~Jiang, and S.~Singh.
\newblock Repeated inverse reinforcement learning.
\newblock \emph{Advances in Neural Information Processing Systems},
  30:\penalty0 1815--1824, 2017.

\bibitem[B{\"a}uerle and Rieder(2014)]{bauerle2014more}
N.~B{\"a}uerle and U.~Rieder.
\newblock More risk-sensitive markov decision processes.
\newblock \emph{Mathematics of Operations Research}, 39\penalty0 (1):\penalty0
  105--120, 2014.

\bibitem[Bjork and Murgoci(2010)]{bjork2010general}
T.~Bjork and A.~Murgoci.
\newblock A general theory of markovian time inconsistent stochastic control
  problems.
\newblock \emph{Available at SSRN 1694759}, 2010.

\bibitem[Bj{\"o}rk and Murgoci(2014)]{bjork2014theory}
T.~Bj{\"o}rk and A.~Murgoci.
\newblock A theory of markovian time-inconsistent stochastic control in
  discrete time.
\newblock \emph{Finance and Stochastics}, 18:\penalty0 545--592, 2014.

\bibitem[Bj{\"o}rk et~al.(2017)Bj{\"o}rk, Khapko, and Murgoci]{bjork2017time}
T.~Bj{\"o}rk, M.~Khapko, and A.~Murgoci.
\newblock On time-inconsistent stochastic control in continuous time.
\newblock \emph{Finance and Stochastics}, 21:\penalty0 331--360, 2017.

\bibitem[Bloem and Bambos(2014)]{bloem2014infinite}
M.~Bloem and N.~Bambos.
\newblock Infinite time horizon maximum causal entropy inverse reinforcement
  learning.
\newblock In \emph{53rd IEEE conference on decision and control}, pages
  4911--4916. IEEE, 2014.

\bibitem[Boularias et~al.(2011)Boularias, Kober, and
  Peters]{boularias2011relative}
A.~Boularias, J.~Kober, and J.~Peters.
\newblock Relative entropy inverse reinforcement learning.
\newblock In \emph{Proceedings of the Fourteenth International Conference on
  Artificial Intelligence and Statistics}, pages 182--189. JMLR Workshop and
  Conference Proceedings, 2011.

\bibitem[Boyd et~al.(1994)Boyd, El~Ghaoui, Feron, and
  Balakrishnan]{boyd1994linear}
S.~Boyd, L.~El~Ghaoui, E.~Feron, and V.~Balakrishnan.
\newblock \emph{Linear matrix inequalities in system and control theory}.
\newblock SIAM, 1994.

\bibitem[Cao et~al.(2021)Cao, Cohen, and Szpruch]{cao2021identifiability}
H.~Cao, S.~Cohen, and L.~Szpruch.
\newblock Identifiability in inverse reinforcement learning.
\newblock \emph{Advances in Neural Information Processing Systems},
  34:\penalty0 12362--12373, 2021.

\bibitem[Capponi and Zhang(2024)]{capponi2024continuous}
A.~Capponi and Y.~Zhang.
\newblock A continuous time framework for sequential goal-based wealth
  management.
\newblock \emph{Management Science}, 2024.

\bibitem[Capponi et~al.(2022)Capponi, Olafsson, and
  Zariphopoulou]{capponi2022personalized}
A.~Capponi, S.~Olafsson, and T.~Zariphopoulou.
\newblock Personalized robo-advising: Enhancing investment through client
  interaction.
\newblock \emph{Management Science}, 68\penalty0 (4):\penalty0 2485--2512,
  2022.

\bibitem[Chewning et~al.(2012)Chewning, Bylund, Shah, Arora, Gueguen, and
  Makoul]{chewning2012patient}
B.~Chewning, C.~L. Bylund, B.~Shah, N.~K. Arora, J.~A. Gueguen, and G.~Makoul.
\newblock Patient preferences for shared decisions: a systematic review.
\newblock \emph{Patient education and counseling}, 86\penalty0 (1):\penalty0
  9--18, 2012.

\bibitem[Christiano et~al.(2017)Christiano, Leike, Brown, Martic, Legg, and
  Amodei]{christiano2017deep}
P.~F. Christiano, J.~Leike, T.~Brown, M.~Martic, S.~Legg, and D.~Amodei.
\newblock Deep reinforcement learning from human preferences.
\newblock \emph{Advances in neural information processing systems}, 30, 2017.

\bibitem[Cox et~al.(2014)Cox, Hobson, and OB{\l}{\'O}J]{cox2014utility}
A.~M. Cox, D.~Hobson, and J.~OB{\l}{\'O}J.
\newblock Utility theory front to back—inferring utility from agents'choices.
\newblock \emph{International Journal of Theoretical and Applied Finance},
  17\penalty0 (03):\penalty0 1450018, 2014.

\bibitem[Dai et~al.(2023)Dai, Dong, and Jia]{dai2023learning}
M.~Dai, Y.~Dong, and Y.~Jia.
\newblock Learning equilibrium mean-variance strategy.
\newblock \emph{Mathematical Finance}, 33\penalty0 (4):\penalty0 1166--1212,
  2023.

\bibitem[Derbaix and Abeele(1985)]{derbaix1985consumer}
C.~Derbaix and P.~V. Abeele.
\newblock Consumer inferences and consumer preferences. the status of cognition
  and consciousness in consumer behavior theory.
\newblock \emph{International Journal of Research in Marketing}, 2\penalty0
  (3):\penalty0 157--174, 1985.

\bibitem[Dong and Wang(2024)]{dong2024towards}
C.~Dong and Y.~Wang.
\newblock Towards generalized inverse reinforcement learning.
\newblock \emph{arXiv preprint arXiv:2402.07246}, 2024.

\bibitem[Dybvig and Rogers(1997)]{dybvig1997recovery}
P.~H. Dybvig and L.~C.~G. Rogers.
\newblock Recovery of preferences from observed wealth in a single realization.
\newblock \emph{The Review of Financial Studies}, 10\penalty0 (1):\penalty0
  151--174, 1997.
\newblock ISSN 08939454, 14657368.
\newblock URL \url{http://www.jstor.org/stable/2962259}.

\bibitem[D’Acunto and Rossi(2021)]{d2021robo}
F.~D’Acunto and A.~G. Rossi.
\newblock \emph{Robo-advising}.
\newblock Springer, 2021.

\bibitem[D’Acunto et~al.(2019)D’Acunto, Prabhala, and Rossi]{d2019promises}
F.~D’Acunto, N.~Prabhala, and A.~G. Rossi.
\newblock The promises and pitfalls of robo-advising.
\newblock \emph{The Review of Financial Studies}, 32\penalty0 (5):\penalty0
  1983--2020, 2019.

\bibitem[Ekeland and Lazrak(2010)]{ekeland2010golden}
I.~Ekeland and A.~Lazrak.
\newblock The golden rule when preferences are time inconsistent.
\newblock \emph{Mathematics and Financial Economics}, 4:\penalty0 29--55, 2010.

\bibitem[El~Karoui and Mrad(2021)]{el2021recover}
N.~El~Karoui and M.~Mrad.
\newblock Recover dynamic utility from observable process: Application to the
  economic equilibrium.
\newblock \emph{SIAM Journal on Financial Mathematics}, 12\penalty0
  (1):\penalty0 189--225, 2021.

\bibitem[El~Karoui et~al.(2019)El~Karoui, Hillairet, and
  Mrad]{el2019construction}
N.~El~Karoui, C.~Hillairet, and M.~Mrad.
\newblock Construction of an aggregate consistent utility, without pareto
  optimality. application to long-term yield curve modeling.
\newblock In \emph{Frontiers in Stochastic Analysis--BSDEs, SPDEs and their
  Applications: Edinburgh, July 2017 Selected, Revised and Extended
  Contributions 8}, pages 169--199. Springer, 2019.

\bibitem[El~Karoui et~al.(2024)El~Karoui, Hillairet, and Mrad]{el2024bi}
N.~El~Karoui, C.~Hillairet, and M.~Mrad.
\newblock Bi-revealed utilities in a defaultable universe: A new point of view
  on consumption.
\newblock \emph{Probability, Uncertainty and Quantitative Risk}, 9\penalty0
  (1):\penalty0 13--34, 2024.

\bibitem[Finn et~al.(2016)Finn, Christiano, Abbeel, and
  Levine]{finn2016connection}
C.~Finn, P.~Christiano, P.~Abbeel, and S.~Levine.
\newblock A connection between generative adversarial networks, inverse
  reinforcement learning, and energy-based models.
\newblock \emph{arXiv preprint arXiv:1611.03852}, 2016.

\bibitem[Fleming and Soner(2006)]{fleming2006controlled}
W.~H. Fleming and H.~M. Soner.
\newblock \emph{Controlled Markov processes and viscosity solutions},
  volume~25.
\newblock Springer Science \& Business Media, 2006.

\bibitem[Fu et~al.(2018)Fu, Luo, and Levine]{fu2017learning}
J.~Fu, K.~Luo, and S.~Levine.
\newblock Learning robust rewards with adverserial inverse reinforcement
  learning.
\newblock In \emph{International Conference on Learning Representations}, 2018.

\bibitem[Garg et~al.(2021)Garg, Chakraborty, Cundy, Song, and
  Ermon]{garg2021iq}
D.~Garg, S.~Chakraborty, C.~Cundy, J.~Song, and S.~Ermon.
\newblock Iq-learn: Inverse soft-q learning for imitation.
\newblock \emph{Advances in Neural Information Processing Systems},
  34:\penalty0 4028--4039, 2021.

\bibitem[Haarnoja et~al.(2017)Haarnoja, Tang, Abbeel, and
  Levine]{haarnoja2017reinforcement}
T.~Haarnoja, H.~Tang, P.~Abbeel, and S.~Levine.
\newblock Reinforcement learning with deep energy-based policies.
\newblock In \emph{Proceedings of the 34th International Conference on Machine
  Learning - Volume 70}, ICML'17, page 1352–1361. JMLR.org, 2017.

\bibitem[Hern{\'a}ndez and Possama{\"\i}(2023)]{hernandez2023me}
C.~Hern{\'a}ndez and D.~Possama{\"\i}.
\newblock Me, myself and i: a general theory of non-markovian time-inconsistent
  stochastic control for sophisticated agents.
\newblock \emph{The Annals of Applied Probability}, 33\penalty0 (2):\penalty0
  1396--1458, 2023.

\bibitem[Ho and Ermon(2016)]{ho2016generative}
J.~Ho and S.~Ermon.
\newblock Generative adversarial imitation learning.
\newblock In \emph{Advances in neural information processing systems}, pages
  4565--4573, 2016.

\bibitem[Hu et~al.(2012)Hu, Jin, and Zhou]{hu2012time}
Y.~Hu, H.~Jin, and X.~Y. Zhou.
\newblock Time-inconsistent stochastic linear--quadratic control.
\newblock \emph{SIAM journal on Control and Optimization}, 50\penalty0
  (3):\penalty0 1548--1572, 2012.

\bibitem[Hu et~al.(2017)Hu, Jin, and Zhou]{hu2017time}
Y.~Hu, H.~Jin, and X.~Y. Zhou.
\newblock Time-inconsistent stochastic linear-quadratic control:
  Characterization and uniqueness of equilibrium.
\newblock \emph{SIAM Journal on Control and Optimization}, 55\penalty0
  (2):\penalty0 1261--1279, 2017.

\bibitem[Jin and Yu~Zhou(2008)]{jin2008behavioral}
H.~Jin and X.~Yu~Zhou.
\newblock Behavioral portfolio selection in continuous time.
\newblock \emph{Mathematical Finance: An International Journal of Mathematics,
  Statistics and Financial Economics}, 18\penalty0 (3):\penalty0 385--426,
  2008.

\bibitem[Kalman(1964)]{Kalman64}
R.~E. Kalman.
\newblock {When Is a Linear Control System Optimal?}
\newblock \emph{Journal of Basic Engineering}, 86\penalty0 (1):\penalty0
  51--60, 03 1964.

\bibitem[Karnam et~al.(2017)Karnam, Ma, and Zhang]{karnam2017dynamic}
C.~Karnam, J.~Ma, and J.~Zhang.
\newblock {Dynamic approaches for some time-inconsistent optimization
  problems}.
\newblock \emph{The Annals of Applied Probability}, 27\penalty0 (6):\penalty0
  3435 -- 3477, 2017.
\newblock \doi{10.1214/17-AAP1284}.
\newblock URL \url{https://doi.org/10.1214/17-AAP1284}.

\bibitem[Keeney and Raiffa(1976)]{keeney_raiffa_1993}
R.~L. Keeney and H.~Raiffa.
\newblock \emph{Decisions with Multiple Objectives: Preferences and Value
  Trade-Offs}.
\newblock Wiley, 1976.

\bibitem[Kim et~al.(2021)Kim, Garg, Shiragur, and Ermon]{Kim2021}
K.~Kim, S.~Garg, K.~Shiragur, and S.~Ermon.
\newblock Reward identification in inverse reinforcement learning.
\newblock In M.~Meila and T.~Zhang, editors, \emph{Proceedings of the 38th
  International Conference on Machine Learning}, volume 139 of
  \emph{Proceedings of Machine Learning Research}, pages 5496--5505. PMLR,
  18--24 Jul 2021.
\newblock URL \url{http://proceedings.mlr.press/v139/kim21c.html}.

\bibitem[Levine et~al.(2011)Levine, Popovic, and Koltun]{levine2011nonlinear}
S.~Levine, Z.~Popovic, and V.~Koltun.
\newblock Nonlinear inverse reinforcement learning with gaussian processes.
\newblock \emph{Advances in neural information processing systems},
  24:\penalty0 19--27, 2011.

\bibitem[Linos(1983)]{lions1983hjbI}
P.~L. Linos.
\newblock Optimal control of diffustion processes and hamilton-jacobi-bellman
  equations part i: the dynamic programming principle and application.
\newblock \emph{Communications in Partial Differential Equations}, 8\penalty0
  (10):\penalty0 1101--1174, 1983.
\newblock \doi{10.1080/03605308308820297}.
\newblock URL \url{https://doi.org/10.1080/03605308308820297}.

\bibitem[Lions(1982)]{lions1982hjb}
P.~L. Lions.
\newblock Optimal stochastic control of diffusion type processes and
  hamilton-jacobi-bellman equations.
\newblock In W.~H. Fleming and L.~G. Gorostiza, editors, \emph{Advances in
  Filtering and Optimal Stochastic Control}, pages 199--215, Berlin,
  Heidelberg, 1982. Springer Berlin Heidelberg.
\newblock ISBN 978-3-540-39517-1.

\bibitem[Musiela and Zariphopoulou(2006)]{musiela2006investments}
M.~Musiela and T.~Zariphopoulou.
\newblock Investments and forward utilities.
\newblock \emph{preprint}, 2006.

\bibitem[Musiela and Zariphopoulou(2007)]{Musiela2007}
M.~Musiela and T.~Zariphopoulou.
\newblock \emph{Investment and Valuation Under Backward and Forward Dynamic
  Exponential Utilities in a Stochastic Factor Model}, pages 303--334.
\newblock Birkh{\"a}user Boston, Boston, MA, 2007.
\newblock ISBN 978-0-8176-4545-8.
\newblock \doi{10.1007/978-0-8176-4545-8_16}.
\newblock URL \url{https://doi.org/10.1007/978-0-8176-4545-8_16}.

\bibitem[Ng et~al.(2000)Ng, Russell, et~al.]{ng2000algorithms}
A.~Y. Ng, S.~Russell, et~al.
\newblock Algorithms for inverse reinforcement learning.
\newblock In \emph{Icml}, volume~1, page~2, 2000.

\bibitem[Nicole and Mohamed(2013)]{EKM3013an}
E.~K. Nicole and M.~Mohamed.
\newblock An exact connection between two solvable sdes and a nonlinear utility
  stochastic pde.
\newblock \emph{SIAM Journal on Financial Mathematics}, 4\penalty0
  (1):\penalty0 697--736, 2013.
\newblock \doi{10.1137/10081143X}.
\newblock URL \url{https://doi.org/10.1137/10081143X}.

\bibitem[Pollak(1968)]{pollack1968consistent}
R.~A. Pollak.
\newblock {Consistent Planning1}.
\newblock \emph{The Review of Economic Studies}, 35\penalty0 (2):\penalty0
  201--208, 04 1968.
\newblock ISSN 0034-6527.
\newblock \doi{10.2307/2296548}.
\newblock URL \url{https://doi.org/10.2307/2296548}.

\bibitem[Reddy et~al.(2019)Reddy, Dragan, and Levine]{reddy2019sqil}
S.~Reddy, A.~D. Dragan, and S.~Levine.
\newblock Sqil: Imitation learning via reinforcement learning with sparse
  rewards.
\newblock \emph{arXiv preprint arXiv:1905.11108}, 2019.

\bibitem[Richesson and Vehik(2010)]{richesson2010patient}
R.~Richesson and K.~Vehik.
\newblock Patient registries: utility, validity and inference.
\newblock \emph{Rare diseases epidemiology}, pages 87--104, 2010.

\bibitem[Richter(1966)]{rickter1966revealed}
M.~K. Richter.
\newblock Revealed preference theory.
\newblock \emph{Econometrica}, 34\penalty0 (3):\penalty0 635--645, 1966.
\newblock ISSN 00129682, 14680262.
\newblock URL \url{http://www.jstor.org/stable/1909773}.

\bibitem[Rossi and Utkus(2020)]{rossi2020benefits}
A.~G. Rossi and S.~P. Utkus.
\newblock Who benefits from robo-advising? evidence from machine learning.
\newblock \emph{Evidence from Machine Learning (March 10, 2020)}, 2020.

\bibitem[Russell(1998)]{russell1998learning}
S.~Russell.
\newblock Learning agents for uncertain environments.
\newblock In \emph{Proceedings of the Eleventh Annual Conference on
  Computational Learning Theory}, pages 101--103, 1998.

\bibitem[Samuelson(1948)]{samuelson1948consumption}
P.~A. Samuelson.
\newblock Consumption theory in terms of revealed preference.
\newblock \emph{Economica}, 15\penalty0 (60):\penalty0 243--253, 1948.
\newblock ISSN 00130427, 14680335.
\newblock URL \url{http://www.jstor.org/stable/2549561}.

\bibitem[Sargent(1978)]{sargent1978estimation}
T.~J. Sargent.
\newblock Estimation of dynamic labor demand schedules under rational
  expectations.
\newblock \emph{Journal of Political Economy}, 86\penalty0 (6):\penalty0
  1009--1044, 1978.

\bibitem[Schlaginhaufen and
  Kamgarpour(2023)]{schlaginhaufen2023identifiability}
A.~Schlaginhaufen and M.~Kamgarpour.
\newblock Identifiability and generalizability in constrained inverse
  reinforcement learning.
\newblock In \emph{International Conference on Machine Learning}, pages
  30224--30251. PMLR, 2023.

\bibitem[Shin and Yu(2021)]{shin2021targeted}
J.~Shin and J.~Yu.
\newblock Targeted advertising and consumer inference.
\newblock \emph{Marketing Science}, 40\penalty0 (5):\penalty0 900--922, 2021.

\bibitem[Strotz(1955)]{strotz1955myopia}
R.~H. Strotz.
\newblock Myopia and inconsistency in dynamic utility maximization.
\newblock \emph{The Review of Economic Studies}, 23\penalty0 (3):\penalty0
  165--180, 1955.
\newblock ISSN 00346527, 1467937X.
\newblock URL \url{http://www.jstor.org/stable/2295722}.

\bibitem[Wang and Yu(2021)]{wang2021robo}
H.~Wang and S.~Yu.
\newblock Robo-advising: Enhancing investment with inverse optimization and
  deep reinforcement learning.
\newblock In \emph{2021 20th IEEE international conference on machine learning
  and applications (ICMLA)}, pages 365--372. IEEE, 2021.

\bibitem[Wulfmeier et~al.(2015)Wulfmeier, Ondruska, and
  Posner]{wulfmeier2015maximum}
M.~Wulfmeier, P.~Ondruska, and I.~Posner.
\newblock Maximum entropy deep inverse reinforcement learning.
\newblock \emph{arXiv preprint arXiv:1507.04888}, 2015.

\bibitem[Yong(2012)]{yong2012time}
J.~Yong.
\newblock Time-inconsistent optimal control problems and the equilibrium hjb
  equation.
\newblock \emph{Mathematical Control and Related Fields}, 2\penalty0
  (3):\penalty0 271--329, 2012.
\newblock ISSN 2156-8472.
\newblock \doi{10.3934/mcrf.2012.2.271}.
\newblock URL
  \url{https://www.aimsciences.org/article/id/0f7d3e02-f039-4ce2-92b6-0328dab062ba}.

\bibitem[Zariphopoulou(2001)]{zariphopoulou2001solution}
T.~Zariphopoulou.
\newblock A solution approach to valuation with unhedgeable risks.
\newblock \emph{Finance and stochastics}, 5:\penalty0 61--82, 2001.

\bibitem[Zeng et~al.(2022)Zeng, Li, Garcia, and Hong]{zeng2022maximum}
S.~Zeng, C.~Li, A.~Garcia, and M.~Hong.
\newblock Maximum-likelihood inverse reinforcement learning with finite-time
  guarantees.
\newblock \emph{Advances in Neural Information Processing Systems},
  35:\penalty0 10122--10135, 2022.

\bibitem[Ziebart(2010)]{ziebart2010modeling}
B.~D. Ziebart.
\newblock \emph{{Modeling Purposeful Adaptive Behavior with the Principle of
  Maximum Causal Entropy}}.
\newblock PhD thesis, Carnegie Mellon University, 2010.

\bibitem[Ziebart et~al.(2008)Ziebart, Maas, Bagnell, Dey,
  et~al.]{ziebart2008maximum}
B.~D. Ziebart, A.~L. Maas, J.~A. Bagnell, A.~K. Dey, et~al.
\newblock Maximum entropy inverse reinforcement learning.
\newblock In \emph{Aaai}, volume~8, pages 1433--1438. Chicago, IL, USA, 2008.

\end{thebibliography}

\appendix

\end{document}